\pdfoutput=1
\RequirePackage{ifpdf}
\ifpdf 
\documentclass[pdftex]{sigma}
\else
\documentclass{sigma}
\fi
\usepackage{mathtools}

\usepackage{tikz}

\numberwithin{equation}{section}

\newtheorem{Theorem}{Theorem}[section]
\newtheorem*{Theorem*}{Theorem}
\newtheorem{Corollary}[Theorem]{Corollary}
\newtheorem{Lemma}[Theorem]{Lemma}
\newtheorem{Proposition}[Theorem]{Proposition}
 { \theoremstyle{definition}
\newtheorem{Definition}[Theorem]{Definition}

\newtheorem{Remark}[Theorem]{Remark} }
\newtheorem{notation}[Theorem]{Notation}

\DeclareMathOperator{\NC}{NC}
\DeclareMathOperator{\PF}{PF}

\DeclareMathOperator{\Cat}{Cat}

\DeclareMathOperator{\nil}{nil}
\DeclareMathOperator{\supp}{supp}
\DeclareMathOperator{\Fix}{Fix}

\newcommand{\op}[1]{\operatorname{#1}}
\newcommand{\lldot}{\ll\hspace{-2.5mm}\cdot\hspace{2.5mm}}
\newcommand{\sqsubsetdot}{\sqsubset\hspace{-3.4mm}\cdot\hspace{3.4mm} }

\newcommand{\ind}{\mathbf{ind}}
\newcommand{\sign}{\boldsymbol{\epsilon}}
\newcommand{\park}{\mathbf{park}}
\newcommand{\fpark}{\mathbf{park\mathrlap{'}}}
\newcommand{\abspark}{\mathbf{park}^{\op{abs}}}
\newcommand{\absfpark}{\mathbf{park'}^{\op{abs}}}

\DeclareMathOperator{\bQ}{\bf Q}
\DeclareMathOperator{\bR}{\bf R}
\DeclareMathOperator{\bN}{\bf N}
\DeclareMathOperator{\bD}{\bf D}

\DeclareMathOperator{\bX}{\bf X}
\DeclareMathOperator{\bI}{\bf I}
\DeclareMathOperator{\bU}{\bf U}
\DeclareMathOperator{\bPsi}{\boldsymbol{\Psi}}
\DeclareMathOperator{\bPhi}{\boldsymbol{\Phi}}

\begin{document}
\allowdisplaybreaks

\newcommand{\arXivNumber}{2209.12540}

\renewcommand{\PaperNumber}{069}

\FirstPageHeading

\ShortArticleName{The Generalized Cluster Complex: Refined Enumeration of Faces}

\ArticleName{The Generalized Cluster Complex: Refined \\ Enumeration of Faces and Related Parking Spaces}

\Author{Theo DOUVROPOULOS~$^{\rm a}$ and Matthieu JOSUAT-VERG\`ES~$^{\rm b}$}

\AuthorNameForHeading{T.~Douvropoulos and M.~Josuat-Verg\`es}

\Address{$^{\rm a)}$~University of Massachusetts at Amherst, USA}
\EmailD{\href{mailto:douvropoulos@math.umass.edu}{douvropoulos@math.umass.edu}}
\URLaddressD{\url{https://people.math.umass.edu/~douvropoulos/}}

\Address{$^{\rm b)}$~IRIF, CNRS, Universit\'e Paris-Cit\'e, France}
\EmailD{\href{mailto:matthieu.josuat-verges@irif.fr}{matthieu.josuat-verges@irif.fr}}
\URLaddressD{\url{https://www.irif.fr/~josuat/}}

\ArticleDates{Received September 27, 2022, in final form September 12, 2023; Published online September 26, 2023}

\Abstract{The generalized cluster complex was introduced by Fomin and Reading, as a~natural extension of the Fomin--Zelevinsky cluster complex coming from finite type cluster algebras. In this work, to each face of this complex we associate a parabolic conjugacy class of the underlying finite Coxeter group. We show that the refined enumeration of faces (respectively, positive faces) according to this data gives an explicit formula in terms of the corresponding characteristic polynomial (equivalently, in terms of Orlik--Solomon exponents). This characteristic polynomial originally comes from the theory of hyperplane arrangements, but it is conveniently defined via the parabolic Burnside ring. This makes a connection with the theory of parking spaces: our results eventually rely on some enumeration of chains of noncrossing partitions that were obtained in this context. The precise relations between the formulas counting faces and the one counting chains of noncrossing partitions are combinatorial reciprocities, generalizing the one between Narayana and Kirkman numbers.}


\Classification{05A15; 05E10; 20F55}

\section{Introduction}

The {\it cluster complex} of a finite type cluster algebra was introduced by Fomin and Zelevinsky~\cite{fominzelevinsky}. It is a simplicial complex, which can be built using {\it almost positive roots} as vertices. It can be viewed as the dual of a corresponding {\it associahedron}. A natural extension is the {\it generalized cluster complex}, defined by Fomin and Reading~\cite{fominreading} via {\it colored almost-positive roots}. Although there is no related cluster algebra nor associahedron in this case, the generalized cluster complex can be given a representation theoretic interpretation via quiver representations~\cite{thomas,zhu}. Most importantly, this is a simplicial complex with nice enumerative and topological properties~\cite{athanasiadistzanaki2,athanasiadistzanaki1,fominzelevinsky,stumpthomaswilliams,tzanaki}. In particular, its number of facets is the {\it Fu{\ss}--Catalan number}
\[
 \Cat^{(m)}(W) := \frac{1}{|W|} \prod_{i=1}^n (m h + e_i + 1),
\]
and its number of positive facets is the {\it positive Fu{\ss}--Catalan number}
\[
 \Cat^{(m)}_+(W) := \frac{1}{|W|} \prod_{i=1}^n (m h + e_i - 1).
\]
Here, $W$ is a finite and irreducible real reflection group, $h$ is its Coxeter number, and $e_1,\dots,e_n$ is the sequence of exponents. Moreover, the {\it Fu{\ss} parameter} $m$ is the number of different colors that a positive root can have. See Section~\ref{sec:defclus} for details.

Consider a {\it flat}, i.e., an element $X$ in the intersection lattice $L(W)$ generated by the reflecting hyperplanes of $W$. There is an associated hyperplane arrangement on $X$ called the {\it restricted arrangement} (a general reference on this subject is Orlik and Terao~\cite{orlikterao}). It turns out to be a {\it free arrangement}, so that its {\it characteristic polynomial} $p_X(t)$ is factorized in the form $\prod_{i=1}^{\dim(X)} \big(t-b_i^X\big)$ where the roots $b_i^X$ are positive integers called the {\it Orlik--Solomon exponents}. (Two flats that are congruent under the action of $W$ have the same characteristic polynomial and we consider that indices are orbits of flats.) It has been established that these characteristic polynomials can be used to refine the enumeration of some Catalan families. In particular, Sommers~\cite{sommers2} considers certain ideals in Lie theory (that correspond to the combinatorial notion of {\it nonnesting partitions}). He showed in~\cite[Theorem~5.7]{sommers2} that the number of such ideals (respectively, positive ideals) is the Catalan number $\Cat(W) := \Cat^{(1)}(W)$\big(respectively, the positive Catalan number \smash{$\Cat_+(W) := \Cat^{(1)}_+(W)$\big)}. Moreover, he showed in~\cite[Proposition~6.6]{sommers2} that the number of ideals (respectively, positive ideals) associated to the orbit of a flat $X$ under a certain natural map is
\[
 \frac{ p_X(h+1) }{[N(W_X):W_X]},
 \qquad \text{respectively,} \qquad
 \frac{ p_X(h-1) }{[N(W_X):W_X]}.
\]
(Here, the denominator is the index of the parabolic subgroup $W_X$ in its normalizer, see Section~\ref{sec:parking} for details.) These are known as {\it Kreweras numbers}. They respectively add up to $\Cat(W)$ and $\Cat_+(W)$ when we sum over all orbits of flats (or, to the {\it Narayana numbers} $\operatorname{Nar}(W,k)$ and $\operatorname{Nar}_+(W,k)$ if we sum over orbits of flats of dimension $k$). The Fu{\ss} generalization also exists: see Rhoades~\cite[Section~9]{rhoades} (it relies on Athanasiadis' notion of co-filtered chain of ideals~\cite{athanasiadis}).

Another Catalan family consists of {\it noncrossing partitions} (see~\cite{baumeisteretal} for a recent survey). They are particularly important here because of their close connection to the cluster complex. A refined enumeration of these objects akin to Sommers' exists: indeed Athanasiadis and Reiner~\cite[Theorem~6.3]{athanasiadisreiner} had previously shown that such refined enumerations of nonnesting and noncrossing partitions coincide (without giving the explicit formulas in terms of characteristic polynomials). However, this coincidence is proved via a case-by-case check using the finite type classification and finding a more conceptual explanation is still an open problem. The most promising attempt in this direction is {\it parking space theory}, introduced by Armstrong, Reiner and Rhoades~\cite{armstrongreinerrhoades}: among various other features, it gives a representation theoretic framework to prove this kind of refined enumeration. In particular, the characteristic polynomials $p_X(t)$ naturally appear there via an identity in the parabolic Burnside ring (see Orlik and Solomon~\cite{orliksolomon}, or Section~\ref{sec:parking}), rather than via hyperplane arrangements. This theory was also extended in the Fu{\ss}--Catalan setting by Rhoades~\cite{rhoades}. This will be reviewed in Section~\ref{sec:parking}. Let us just mention here that $m$-element chains (respectively, full support $m$-element chains) of noncrossing partitions are counted by $\Cat^{(m)}(W)$ \big(respectively, $\Cat_+^{(m)}(W)$\big) , and a natural refinement gives the numbers:
\begin{align} \label{eq:krewm}
 \frac{p_X(mh+1)}{[N(W_X):W_X]},
 \qquad\text{respectively,}\qquad
 \frac{p_X(mh-1)}{[N(W_X):W_X]}.
\end{align}
See Theorems~\ref{theo_park} and~\ref{theo_fpark} for the precise statements.

The entries of the $f$-vector of the cluster complex (i.e., the number of faces of a given dimension) are known as the {\it Kirkman numbers}. It is known that Kirkman numbers and Narayana numbers are related by a combinatorial reciprocity, and one of our goals is to refine this phenomenon: in Corollary~\ref{formulas_gammagammaplus}, we show that the number of faces (respectively, positive faces) associated to the orbit of a flat $X$ under a certain natural map is
\begin{align} \label{eq:introformulagamma}
 (-1)^{\dim(X)} \frac{p_X(-mh-1)}{[N(W_X):W_X]},
 \qquad\text{respectively,}\qquad
 (-1)^{\dim(X)} \frac{p_X(-mh+1)}{[N(W_X):W_X]}.
\end{align}
The meaning of ``combinatorial reciprocity'' is that these numbers are related to those in~\eqref{eq:krewm} by $m\leftrightarrow -m$. (The combinatorial reciprocity mentioned above between Narayana and Kirkman numbers is recovered by summing over orbit of flats of a given dimension $k$.) The formulas in~\eqref{eq:introformulagamma} gives our refined enumeration of the faces of the generalized cluster complex. The derivation is case-free, but eventually relies on Theorem~\ref{theo_park} mentioned above (for which the only known proof is via a case-by-case check).

There are a few preliminaries that are interesting on their own, as well as nice consequences. We outline both by giving the detailed organization of this article:
\begin{itemize}\itemsep=0pt
 \item Section~\ref{sec:background} contains some background material.
 \item {\it Parking space theory} is reviewed in Section~\ref{sec:parking}, following~\cite{armstrongreinerrhoades,rhoades}. The aim is to present Theorems~\ref{theo_park} and~\ref{theo_fpark}, on which our results rely. Sections~\ref{parkingproof1} and~\ref{parkingproof2} contain proofs of the latter theorem (more precisely, we show the equivalence between Theorems~\ref{theo_park} and~\ref{theo_fpark}).\looseness=1
 \item The {\it generalized cluster complex} is reviewed in Section~\ref{sec:defclus}, following~\cite{bradywatt,fominreading,fominzelevinsky,tzanaki}. This section also contains the definition of the natural map from faces to orbits of flats. The formulas in terms of characteristic polynomials are obtained in Section~\ref{sec:reciprocities}, via combinatorial reciprocities which make a link with chains of noncrossing partitions. It is also proved in Section~\ref{sec:bij} via a bijection, again making a link with certain chains of noncrossing partitions.
 In Section~\ref{sec:fhvectors}, we give some consequences concerning the $f$-vectors and $h$-vectors of the generalized cluster complex. In particular, an identity can be seen as a refinement of the relations between $f$- and $h$-vectors. Finally, Section~\ref{sec:recfaces} provides a recursion satisfied by the left-hand side of~\eqref{eq:introformulagamma}, proved via the combinatorics of the generalized cluster complex.\looseness=1
 \item {\it Minimal factorizations} of the Coxeter element are ubiquitous in the context of noncrossing partitions and cluster complexes. In Section~\ref{sec:factor}, we get a formula for a $q$-enumeration of certain minimal factorizations (where one factor is in a given parabolic conjugacy class and the others are reflections). This is related to the enumeration of faces of the generalized cluster complex, via a recursion which is equivalent to that in Section~\ref{sec:recfaces}.
 \item {\it Two order relations} on noncrossing partitions, denoted $\sqsubset$ and $\ll$, are used throughout. We introduced them in~\cite{bianejosuat}, and in some sense they refine the absolute order (used to define the lattice structure on noncrossing partitions). They are useful to prove an identity on parking spaces (Section~\ref{parkingproof2}), to give bijections between faces of the generalized cluster complex and certain chains of noncrossing partitions (Section~\ref{sec:bij}), and to define the $q$-statistic in our $q$-enumeration of minimal factorizations mentioned above (Section~\ref{sec:factor}).
\end{itemize}

\section{Preliminary definitions}\label{sec:background}

Through this work, $W$ is a finite real reflection group of rank $n$. We don't assume that it be irreducible, unless stated otherwise. Its geometric representation is an $n$-dimensional Euclidean space $V \simeq \mathbb{R}^n$. Let $T \subset W$ denote the set of reflections, and $S\subset T$ a set of simple reflections that we write $S = \{s_1, \dots , s_n\}$. The standard parabolic subgroup $W_I$ for $I\subset S$ is the subgroup of $W$ generated by $I$. Any subgroup conjugate to some $W_I$ for $I\subset S$ is called a parabolic subgroup. The {\it support} of $w\in W$ is
\[
 \supp(w) := \min\big\{ I\subset S\colon w\in W_I \big\},
\]
where the minimum is taken with respect to inclusion, and $w$ is said to have {\it full support} if $\supp(w)=S$.

\subsection{The intersection lattice}

\begin{Definition}
For each $w\in W$, we denote
\[
 \Fix(w) := \ker(w-I) = \big\{ v\in V\colon w(v)=v \big\}.
\]
The \emph{intersection lattice} of $W$ is defined as
\[
 L(W) := \big\{ \Fix(w) \colon w\in W \big\}.
\]
By convention, the order relation on $L(W)$ is reverse inclusion.
\end{Definition}

In type $A_{n-1}$, this is the lattice of set partitions of $\{1,\dots,n\}$ ordered by refinement.

The poset $L(W)$ is a lattice, and the join operation is given by intersection of subspaces. It is order-isomorphic to the lattice of parabolic subgroups of $W$ (where the order is inclusion) via
\begin{align} \label{def_WX}
 X \mapsto W_X := \{ w\in W \colon X\subset \Fix(w) \}.
\end{align}
There is a natural action of $W$ on $L(W)$ induced by the action of $W$ on $V$, and the corresponding action on parabolic subgroups is by conjugation. Let $\sim$ denote the equivalence class given by the orbit decomposition.

The apparent clash of notation between $W_I$ and $W_X$ is dealt with by identifying the powerset of $S$ with a subposet of $L(W)$, via
\[
 I \mapsto \Fix\Big( \prod_{s\in I} s \Big).
\]
(The order of the product is given by a fixed total order on $S$.) Implicitly, it is assumed that~$I$,~$J$,~$K$ are in this subposet, while~$X$,~$Y$,~$Z$ are general flats in $L(W)$. In particular, $I \sim J$ means that the standard parabolic subgroups $W_I$ and $W_J$ are conjugate. This convention will be used when some objects are indexed sometimes by $I$, $J$, $K$ and sometimes by $X$, $Y$, $Z$.

\subsection{The noncrossing partition lattice}

Noncrossing partitions are defined with respect to a {\it standard Coxeter element} $c$, which is the product of all the simple reflections in $S$. By reindexing the set $S=\{s_1,\dots,s_n\}$, we can assume $c=s_1 \cdots s_n$.

\begin{Definition}
 The {\it reflection length} of $w\in W$ is defined by
 \[
 \ell(w) := n-\dim(\Fix(w)).
 \]
 It is also the minimal integer $k$ such that $w$ is a product of $k$ reflections, and any such factorization $w=t_1 \cdots t_k$ is called {\it minimal}.
\end{Definition}

\begin{Definition}
 The {\it absolute order} of $W$ is defined by $w_1 \leq w_2 $ if $\ell(w_1) + \ell\big(w_1^{-1}w_2\big) = \ell(w_2)$. Alternatively, $w_1 \leq w_2 $ iff a minimal reflection factorization of $w_1$ is a subword of a minimal reflection factorization of $w_2$. The {\it noncrossing partition lattice} of $W$ (with respect to a standard Coxeter element $c$), denoted $\NC(W,c)$, is defined as the order ideal containing elements below~$c$ in the absolute order.
\end{Definition}

This is a widely studied object, and we refer to Baumeister et~al.~\cite{baumeisteretal} for a recent survey. Note that we have a map $w\mapsto \Fix(w)$ from $\NC(W,c)$ to $L(W)$. This map is injective, increasing and rank-preserving. Moreover, $w_1,w_2 \in \NC(W,c)$ are conjugate (in $W$) iff $\Fix(w_1)$ and $\Fix(w_2)$ are in the same orbit under the action of $W$.

For each integer $m\geq 1$, we can define the {\it Fu{\ss}--Catalan numbers} and their {\it positive} counterpart:
\begin{gather*}
 \Cat^{(m)}(W)
:=
 \#
 \big\{
 (\pi_i)_{1\leq i \leq m} \in \NC(W,c)^m
 \colon \pi_1 \leq \dots \leq \pi_m
 \big\},\\
 \Cat_+^{(m)}(W)
:=
 \#
 \big\{
 (\pi_i)_{1\leq i \leq m} \in \NC(W,c)^m
 \colon \pi_1 \leq \dots \leq \pi_m,\; \supp(w_m)=S
 \big\}.
\end{gather*}
In the case where $W$ is irreducible, the formulas given in the introduction in terms of exponents and the Coxeter number hold. Otherwise, we have
\begin{gather*}
 \Cat^{(m)}\Bigg( \prod_{i=1}^j W_i \Bigg)
 =
 \prod_{i=1}^j \Cat^{(m)}\big( W_i \big),\\
 \Cat_+^{(m)}\Bigg( \prod_{i=1}^j W_i \Bigg)
 =
 \prod_{i=1}^j \Cat_+^{(m)}\big( W_i \big).
\end{gather*}
The rank-refined enumeration of noncrossing partitions gives {\it Narayana numbers}. One can also define Fu{\ss}--Narayana numbers by (see~\cite{armstrong}):
\begin{gather} \label{def:narayana}
 \operatorname{Nar}^{(m)}(W,i)
 :=
 \#
 \big\{\!
 (\pi_i)_{1\leq i \leq m} \in \NC(W,c)^m
 \colon \pi_1 \leq\! \cdots \! \leq \pi_m,\, \ell(\pi_1)=i
 \big\},\\
 \label{def:narayanap}
 \operatorname{Nar}_+^{(m)}(W,i)
 :=
 \#
 \big\{\!
 (\pi_i)_{1\leq i \leq m} \in \NC(W,c)^m
 \colon \pi_1 \leq\! \cdots \! \leq \pi_m,\, \ell(\pi_1)=i,\, \supp(w_1)=S
 \big\}.\!\!\!\!\!\!
\end{gather}

\begin{Lemma} \label{lemm:bij_xi}
 The following two sets are in bijection:
 \begin{itemize}\itemsep=0pt
 \item \emph{parabolic conjugacy classes}, i.e., conjugacy classes $\mathcal{X}\subset W$ such that $\mathcal{X}\cap \NC(W,c) \neq \varnothing$,
 \item $L(W)/W$, i.e., orbits of flats under the action of $W$.
 \end{itemize}
\end{Lemma}

\begin{proof}
We refer to \cite{orliksolomon} (see Lemma~(3.4), Lemma~(3.5), and the lines thereafter). Let us briefly describe the explicit bijections.

To each parabolic conjugacy class $\mathcal{X}$, we associate the orbit of $\Fix(w)$ for some arbitrary $w\in\mathcal{X}$. In the other direction, consider a flat $X$ defined up to the action of $W$. There is $I\subset S$ such that $W_X$ is conjugate to the standard parabolic subgroup $W_I$, and to the orbit of $X$ we associate the conjugacy class of $\prod_{s\in I}s$ (the order of the product is irrelevant).
\end{proof}

For $X \in L(W)$, the parabolic conjugacy class corresponding to its orbit in $L(W)/W$ (via the previous bijection) is denoted $\mathcal{X}$. We sometimes use the bijection implicitly and the notation makes clear what are the objects. For example, if $w\in\NC(W,c)$ then the condition $w\in\mathcal{X}$ is equivalent to $\Fix(w)\sim X$.

In type $A_{n-1}$, every conjugacy class contains a noncrossing partition. The sets in the previous lemma identify with the set of integer partitions of $n$.

\begin{Lemma} 
 For each $w\in \NC(W,c)$, the parabolic subgroup $W_{\Fix(w)}$ is the minimal parabolic subgroup of $W$ containing $w$.
\end{Lemma}

\begin{Proposition} \label{prop:wcox}
 A noncrossing partition $w\in\NC(W,c)$ is a standard Coxeter element of $W_{\Fix(w)}$. This means there is a factorization $w = t_1 \cdots t_k$ (called the \emph{canonical factorization}, unique up to commutations among the factors) where the elements $t_1, \dots , t_k$ are the simple generators of $W_{\Fix(w)}$.
\end{Proposition}

This has been observed by several authors. We refer to \cite[Proposition~3.1]{bianejosuat} for a discussion.

\subsection{A tale of two orders}

The lattice structure of $\NC(W,c)$ (the absolute order) can be refined: we introduced in~\cite{bianejosuat} two partial order relations $\sqsubset$ and $\ll$ with many combinatorial properties. In particular, they are useful to deal with the combinatorics of the cluster complex.

\begin{Definition}[\cite{bianejosuat}]
Let $w\in \NC(W,c)$, and write its canonical factorization $w = t_1 \cdots t_k$. We define $\sqsubset$ and $\ll$ on $\NC(W,c)$ by:
\begin{itemize}\itemsep=0pt
 \item $v\sqsubset w$ if $v$ can be written as a subword of $t_1 \cdots t_k$ (so that $v\leq w$, in particular),
 \item $v\ll w$ if $v\leq w$ and $v$ has full support in $W_{\Fix(w)}$, i.e., each $t_i$ for $1\leq i \leq k$ appears at least once in any factorization $v = t_{i_1} t_{i_2} \cdots$.
\end{itemize}
\end{Definition}

Note that these partial orders are such that $v\sqsubset w \Rightarrow v\leq w$ and $v\ll w \Rightarrow v\leq w$.

Another characterization of these partial orders makes a close connection with the {\it Bruhat order}, denoted $\leq_B$. This will be used in Section~\ref{sec:factor}. It states that the cover relations for~$\sqsubset$ and~$\ll$ are such that
\[
 u \sqsubsetdot v \;\Leftrightarrow\; u\lessdot v \quad\text{and}\quad u\leq_B v, \qquad
 u \lldot v \;\Leftrightarrow\; u\lessdot v \quad\text{and}\quad u\geq_B v.
\]

Let us give some other properties, mostly taken from~\cite{bianejosuat}.

\begin{Proposition} \label{prop:order_uvw}
 For each $u,w \in\NC(W,c)$ such that $u \leq w$, there exists a unique $v \in\NC(W,c)$ such that $u \ll v \sqsubset w$.
\end{Proposition}

(See also~Lemma~\ref{lemm:intervalsfaces} for a related result.)

\begin{proof}
 Let us first consider the case where $w$ is maximal, i.e., it is the Coxeter element $c$. In this case, let $I \subset S$ be the support of $u$ and $v$ be the unique $v\sqsubset c$ which is the product of $s_i$ for~$i\in I$. It is easily checked that it satisfies $u \ll v \sqsubset w$, and it is unique. The general case follows by doing the same procedure in the parabolic subgroup $W_{\Fix(w)}$.
\end{proof}

\begin{Proposition}[{\cite[Corollary~4.10]{bianejosuat}}] \label{lemm:bool}
 Let $w\in\NC(W,c)$. We have
 \begin{itemize}\itemsep=0pt
 \item The two orders $\leq$ and $\sqsubset$ agree on the set $\{ v\in\NC(W,c)\colon v\sqsubset w \}$. The resulting poset is a boolean lattice of order $2^{\ell(w)}$, containing all elements that can be written as subwords of the canonical factorization of $w$.
 \item The two orders $\leq$ and $\ll$ agree on the set $\{ v\in\NC(W,c)\colon v\gg w \}$. The resulting poset is a boolean lattice of order $2^{\#\supp(w) - \ell(w)}$, and its maximal element is the unique $w'$ such that $w\ll w' \sqsubset c$ (given by the previous proposition).
 \end{itemize}
\end{Proposition}

\section{Parking spaces and their characters}
\label{sec:parking}

The goal of this section is to introduce some background on parking space theory, as well as a~``prime'' analog. The main result about the prime parking space is Theorem~\ref{theo_fpark}.

\begin{Remark}
 The terminology is not well-established, but there does exist a notion of ``prime parking function''. See~\cite{duarte} for a proof that there are $(n-1)^{n-1}$ prime parking functions of length~$n$. Our prime parking space is the natural analog in the Coxeter setting. The adjective ``positive'' as in the positive Fu{\ss}--Catalan number $\Cat^{(m)}_+(W)$ is natural in cluster theory (see Section~\ref{sec:defclus}). The term {\it ``Fu{\ss}--Dogolon''} has also been coined as a prime/positive analog of ``Fu{\ss}--Catalan''.
\end{Remark}

\subsection{The parabolic Burnside ring}

We first need some preliminaries about characters of $W$, due to Orlik and Solomon~\cite{orliksolomon}. See also Geck and Pfeifer~\cite[Chapter~2.4]{geckpfeiffer}.

Generically, we will use bold symbols to denote characters. In particular, ${\bf 1}$ denotes the trivial character (of a group which is clear from the context).

\begin{Definition}
 For $I\subset S$, let $ \bPhi_I := \ind_{W_I}^W(\bf 1)$ (the trivial character of $W_I$ induced to $W$). The {\it parabolic Burnside ring} $R(W)$ of $W$ is the ring linearly generated by $(\bPhi_I)_{I\subset S}$ (as a subring of the character ring of $W$).
\end{Definition}

It is not obvious that the linear span of $(\bPhi_I)_{I\subset S}$ is indeed a ring. We refer to~\cite[Section~2.4.3]{geckpfeiffer}. It also follows from {\it loc.~cit.}~that a basis of $R(W)$ is $(\bPhi_I)_{I\in \Theta}$, where $\Theta$ is a set of representatives of subsets of $S$ modulo the equivalence relation $\sim$ as in Lemma~\ref{lemm:bij_xi}. Using implicitly one of the bijections from Lemma~\ref{lemm:bij_xi}, we identify $\Theta$ with a set of representatives for the quotient $L(W)/W$. We thus write $\bPhi_X$ in place of $\bPhi_I$ for $X\in L(W)$ such that $W_X$ and $W_I$ are conjugate.

Note that $\bPhi_{S}$ is the trivial character of $W$, and the unit of $R(W)$. Also, it can be seen that~$\bPhi_I$ is the character of the representation $\mathbb{C}^{ W / W_I }$ (the linearization of the group action on~$W/ W_I$ where $W$ acts by left multiplication on the cosets).

In type $A_n$, $R(W)$ is the space of degree $n+1$ symmetric functions under the Kronecker product.

\begin{Remark}
 Let us mention that $R(W)$ is a subring of the {\it Burnside ring} of $W$ (which is linearly generated by characters of the representations $\mathbb{C}^{ W / W' }$ where $W'$ is any subgroup of $W$). The terminology comes from the fact that here we only consider parabolic subgroups.
\end{Remark}

\begin{Remark} \label{carac_OS}
 The algebra $\mathbb{Q}\otimes R(W)$ is the space of functions $\chi \colon W \to \mathbb{Q}$ such that the value~$\chi(w)$ only depends on the orbit of $\Fix(w)$ in $L(W)/W$. This characterization is essentially due to Orlik and Solomon~\cite{orliksolomon}.
\end{Remark}

The {\it sign character} of $W$, denoted $\sign$, is defined by
\[
 \sign(w) := (-1)^{\ell(w)} = (-1)^{n-\dim\Fix(w)}.
\]
This character is usually defined with the Coxeter length rather than the reflection length, but they have the same parity as each reflection has odd Coxeter length. The following lemma shows how it acts on $R(W)$ by multiplication.

\begin{Lemma}[Solomon~\cite{solomon}] \label{lemm_solomon}
For any $J\subset S$, we have
\begin{equation}\label{epstensor}
 \sign \otimes \bPhi_J = \sum_{I\subset J} (-1)^{\#I} \bPhi_I.
\end{equation}
\end{Lemma}

Note that by taking $\bPhi_J$ to be the trivial character (i.e., $J=S$), the previous equation gives~$\sign$ as an alternating sum of $\bPhi_I$, showing in particular that $\sign\in R(W)$.

Now, define a class function $\bPsi_t\colon W \to \mathbb{Z}$ for each integer $t$ by
\[
 \bPsi_t(w) = t^{\dim \Fix(w)}.
\]
In particular, $\bPsi_{-1} = (-1)^n \sign$.
Note that $\bPsi_t \in \mathbb{Q} \otimes R(W)$ for each $t\in \mathbb{Z}$. This is easily seen from the characterization of $\mathbb{Q} \otimes R(W)$ stated in Remark~\ref{carac_OS}.

For special values of $t$, in particular $t=mh+1$ with $m$ a non-negative integer and $t=mh-1$ with $m$ a positive integer (these are the ones relevant to our work), Gordon and Griffeth \cite{gordongriffeth} show that the class function $\bPsi_t$ is in fact a character of $W$ associated to the rational Cherednik algebra for $W$ with parameter $t/h$. This construction also goes through when $t$ is relatively prime to $h$, but there the Cherednik character and the class function $\bPsi_t$ differ by a Galois twist (which depends on $t$ but is trivial when $W$ is real and $t=mh\pm 1$). In some sense, even though it fails to be a permutation character (which is what we need in this work) the Cherednik approach has proven better for the study of \emph{rational parking functions} as in \cite{GLTW}.

\begin{Remark}
 In the more general context of complex reflection groups, Ito and Okada~\cite{itookada} gave a classification of positive integers $t$ such that $\bPsi_t$ is the character of a representation or a~permutation representation. For any finite complex reflection group with Coxeter number $h$, it is still true that $\bPsi_{mh+1}$ ($m\geq 0$) is the genuine character of a permutation representation. But this does not hold for $\bPsi_{mh-1}$ ($m>0$) beyond the real case.
\end{Remark}

In the next statement, we treat $t$ as a formal variable rather than an integer. Also, recall from the introduction that $N(W_X)$ is the normalizer of $W_X$ in $W$.

\begin{Proposition}[Orlik and Solomon~\cite{orliksolomon}] \label{prop:psik_expansion}
In $\mathbb{Q}[t] \otimes R(W)$, there is an expansion
\begin{equation} \label{psik_expansion}
 \bPsi_t = \sum_{ X \in \Theta} \frac{p_X(t)}{[N(W_X):W_X]}
 \bPhi_X,
\end{equation}
where $p_X(t)$ is a polynomial in $t$ called the \emph{characteristic polynomial} of $X$. It can be defined using the Möbius function $\mu$ of $L(W)$ by
\begin{equation} \label{eq:def_chi}
 p_X(t)
 =
 \sum_{Y\in L(W),\; Y\geq X} \mu(X,Y) t^{\dim (Y)}
\end{equation}
and can be factorized in the form
\[
 p_X(t) = \prod_{i=1}^{\dim(X)}\big(t-b_i^X \big),
\]
where the roots $b_i^X$ are positive integers called the \emph{Orlik--Solomon exponents} of $X$.
\end{Proposition}

See~\cite{orlikterao} for tables containing the Orlik--Solomon exponents for all irreducible $W$ in the finite type classification. When $X$ is the minimal element of $L(W)$ (i.e., the full-dimensional subspace of the geometric representation $V$ of $W$), the associated Orlik--Solomon exponents are the exponents $e_1,\dots,e_n$ of $W$ (classically defined by considering eigenvalues of the Coxeter element, see~\cite{humphreys}).

From the previous proposition, we see that there are at least two ways to compute the characteristic polynomials (or the Orlik--Solomon exponents, by taking their roots):
\begin{itemize}\itemsep=0pt
 \item we can use the intersection lattice and its Möbius function via equation~\eqref{eq:def_chi},
 \item we can use equation~\eqref{psik_expansion} and character calculations.
\end{itemize}
Let us make the second point more explicit. The values of the characters $\bPhi_I$ on parabolic conjugacy classes can be organized in a square matrix (with rows and columns indexed by $\Theta$) called the {\it parabolic table of marks}. An algorithm to compute it is given by Geck and Pfeiffer~\cite[Section~2.4]{geckpfeiffer}. Since the values of the character $\bPsi_t$ is explicit, we can get the coefficients $\big([N(W_I):W_I]^{-1} p_I(t)\big)$ by inverting the parabolic table of marks.

Another way to compute these characteristic polynomials is given by Sommers in~\cite[Propositions~4.7 and~5.1]{sommers1}. We will give another method below, by giving a recursion satisfied by the numbers $\gamma(W,\mathcal{X},m)$ (see Section~\ref{sec:recfaces}).

The denominator $[ N(W_X) : W_X ]$ that appears above can be written differently. Orlik and Solomon~\cite{orliksolomon} showed that
\begin{equation} \label{factornu}
 \frac{1}{[ N(W_X) : W_X ]} = \frac{ \nu(X) }{ \prod_{i=1}^k \big(b_i^X + 1\big)},
\end{equation}
where $\nu(X)$ is the number of $J\subset S$ such that $W_J \sim W_X$. Indeed, this follows by plugging $t=-1$ in the previous proposition, and using Lemma~\ref{lemm_solomon}. As a check of what happens in the case of the symmetric group (type A), let $\lambda$ be a partition of $n$ and $\mathfrak{S}_\lambda$ the corresponding Young subgroup of $\mathfrak{S}_n$. Let $\mu_i$ be the multiplicity of $i$ in $\lambda$. The normalizer $N(\mathfrak{S}_\lambda)$ is a semidirect product $\mathfrak{S}_\lambda \rtimes (\prod_i \mathfrak{S}_{\mu_i} )$, and it follows $[N(\mathfrak{S}_\lambda):\mathfrak{S}_\lambda] = \prod_i \mu_i!$.

Finally, let us mention the following statement (see also \cite[Theorem~7.4.2]{haiman}).

\begin{Proposition}
If $t>0$ is such that $\bPsi_t$ is the character of a representation, the multiplicity of the trivial character in $\bPsi_t$ is the \emph{rational Catalan number}
\[
 \frac{1}{|W|} \prod_{i=1}^n(t+e_i).
\]
\end{Proposition}

\begin{proof}
The trivial character is orthogonal to the other irreducible characters, and it follows that this multiplicity is
\[
 \frac{1}{|W|} \sum_{w\in W} \bPsi_t (w).
\]
The result then follows from the Shephard--Todd formula~\cite{shephardtodd}
\[
 \sum_{w\in W} t^{\dim \Fix(w)} = \prod_{i=1}^n (t+e_i).
\tag*{\qed}
\]
\renewcommand{\qed}{}
\end{proof}

In the two cases which are relevant to this work, we get that the multiplicity of the trivial character in $\Psi_{mh+1}$ and $\Psi_{mh-1}$ are respectively $\Cat^{(m)}(W)$ and $\Cat_+^{(m)}(W)$.

\subsection{Parking spaces}

We review the theory of parking spaces, introduced by Armstrong, Reiner and Rhoades~\cite{armstrongreinerrhoades}. This is a very brief account, which mostly aims at some enumeration formulas for chains of noncrossing partitions. In particular, we simply consider characters of $W$ whereas the general theory deals with characters of the product of $W$ with a cyclic group. Also, we focus on the {\it noncrossing parking space} but other kinds of parking spaces exist.

Assume that $W$ is irreducible, and let $h$ be its Coxeter number. A {\it parking space} is a representation of $W$ having $\bPsi_{h+1}$ as its character. The name comes from the fact that in type $A_n$, the symmetric group $\mathfrak{S}_{n+1}$ acting on parking functions of length $n+1$ is such a parking space. More generally, a $t$-parking space (for an integer $t\geq 1$) is a representation of $W$ having $\bPsi_{t}$ as its character. In the case where $W$ is a Weyl group, there is a root lattice $Q$ acted on by $W$. It follows from Sommers~\cite[Proposition~3.9]{sommers1} that the quotient $Q/tQ$ is a $t$-parking space for the natural action of $W$, upon some conditions on $t$ (being relatively prime to $h$ is sufficient). See also Haiman~\cite[Section~7.3]{haiman}. For the symmetric group $\mathfrak{S}_n$, other $t$-parking spaces are given by the action on {\it rational parking functions} (see~\cite{armstrongloehrwarrington} for details). See also~\cite{GLTW} for rational parking functions in general type.

Two $t$-parking spaces are isomorphic as representations of $W$, since by definition they have the same character. However, it might be difficult to find an explicit isomorphism and it is therefore interesting to consider various kinds of parking spaces. The {\it noncrossing parking space} from~\cite{armstrongreinerrhoades} is defined in terms of noncrossing partitions. Rhoades' generalization~\cite{rhoades} in the Fu{\ss}--Catalan setting is
\begin{equation} \label{parkbigoplus}
 \bigoplus_{\substack{ w_1,\dots, w_m \in\NC(W,c) \\ w_1\leq \dots \leq w_m }}
 \mathbb{C}^{ W / W_{\Fix(w_1)} }.
\end{equation}
It can be viewed as the linearization of the $W$-set of $m$-{\it parking functions}
\begin{align*} 
 \PF(W,m) := \biguplus_{\substack{ w_1,\dots, w_m \in\NC(W,c) \\ w_1\leq \dots \leq w_m }}
 \{(w_1,\dots,w_m ) \} \times W / W_{\Fix(w_1)}
\end{align*}
(where $W$ acts on the second factor). Rather than the representation itself, we mostly consider its character
\begin{align}
 \label{parkdef}
 \park_{W,m} := \sum_{X\in \Theta} \kappa(W,X,m) \bPhi_X \in R(W),
\end{align}
where
\begin{align} \label{def:kappa}
 \kappa(W, X, m) := \#\big\{ (w_1,\dots,w_m) \in \NC(W,c)^m \colon w_1\leq \dots \leq w_{m} \text{ and } \Fix(w_1) \sim X \big\}.
\end{align}

\begin{Remark}
 The integers $\kappa(W, X, m)$ are known as {\it Kreweras numbers}. Note that they refine Narayana numbers, as we have (as a consequence of the combinatorial definition):
 \begin{align} \label{eq:nar_krew}
 \operatorname{Nar}^{(m)}(W,k)
 =
 \sum_{\substack{ X\in \Theta, \\ \dim(X)=n-k}} \kappa(W,X,m).
 \end{align}
\end{Remark}

\begin{Theorem}[\cite{rhoades}, except for $E_7$, $E_8$] \label{theo_park}
Assume that $W$ is irreducible, with Coxeter number~$h$. Then, the representation in~\eqref{parkbigoplus} is an $(mh+1)$-parking space, i.e., $\park_{W,m} = \bPsi_{mh+1}$. Equivalently,
\begin{equation} \label{eq:formula_kappa}
 \forall X \in \Theta,
 \quad
 \kappa(W,X,m) = \frac{p_X(mh+1)}{[N(W_X):W_X]}.
\end{equation}
\end{Theorem}

\begin{proof}
Note that the equivalence between the two statements follows from Proposition~\ref{prop:psik_expansion}. Since the infinite families of the finite type classification have been treated (see~\cite[Table~3]{rhoades}), there remains only to deal with a finite number of exceptional groups. This can be done by computer.
\end{proof}

Let us give some comments about the computer verification mentioned in this proof. First note that the computation is possible because we have here a finite statement about polynomials in $m$ (whereas Rhoades' conjectures~\cite{rhoades} involves a cyclic group of order $mh$ and is therefore of a different nature). A naive computation of the set $\NC(W,c)$ might be lengthy, but an efficient way to get all noncrossing partitions is to use the characterization in~\cite[Section~4.5]{bianejosuat}. It says that there is a binary relation $\between$ on $T$ (that depends on $c$) such that each $w\in\NC(W,c)$ can be identified with a set of pairwise-related reflections (explicitly, this is the set of simple reflections of $W_{\Fix(w)}$). It gives a quick way to get all noncrossing partitions, and for each $w\in \NC(W,c)$ we also have the canonical factorization $w=t_1 \cdots t_k$ into simple reflections of the parabolic subgroup $W_{\Fix(w)}$. Now, the right-hand side of~\eqref{def:kappa} can be rewritten
\[
 \sum_{ w \in \mathcal{X} } \Cat^{(m-1)}\big(W_{\Fix(w^{-1}c)}\big).
\]
Using the formula for Fu{\ss}--Catalan numbers in terms of exponents \big(which we can use here since we have the Coxeter type of $W_{\Fix(w^{-1}c)}$ via the canonical factorization of $w^{-1}c$\big), we can compute this sum in reasonable time even for $E_7$ and $E_8$.

Now, let $\NC'(W,c) \subset \NC(W,c)$ denote the subset of elements with full support. The ``prime'' analog of the noncrossing parking space is now defined as
\begin{equation} \label{fparkbigoplus}
 \bigoplus_{\substack{ w_1,\dots, w_{m-1} \in\NC(W,c), \; w_m\in \NC'(W,c) \\ w_1\leq \dots \leq w_m }}
 \mathbb{C}^{ W / W_{\Fix(w_1)} }.
\end{equation}
Again, we mostly consider its character
\begin{equation}
 \label{fparkdef}
 \fpark_{W,m} := \sum_{X\in \Theta} \kappa^+(W,X,m) \bPhi_X,
\end{equation}
where
\begin{gather*}
 \kappa^+(W, X, m) := \#\big\{ (w_1,\dots,w_m) \in \NC(W,c)^{m-1} \times \NC'(W,c) \colon \\
 \hphantom{\kappa^+(W, X, m) := \#\big\{}{}\ w_1\leq \dots \leq w_{m} \ \text{and} \ \Fix(w_1) \sim X \big\}.
\end{gather*}

\begin{Remark}
In complete analogy with~\eqref{eq:nar_krew}, we have
 \begin{align*} 
 \operatorname{Nar}_+^{(m)}(W,k)
 =
 \sum_{X\in \Theta, \; \dim(X)=n-k} \kappa^+(W,X,m).
 \end{align*}
\end{Remark}

In complete analogy with Theorem~\ref{theo_park}, we have Theorem~\ref{theo_fpark} below. Even though it hasn't explicitly appeared in the literature, it is not particularly surprising: some known enumeration formulas clearly suggest that the parameter $mh-1$ should appear in the present situation (see~\cite{athanasiadistzanaki2}, concerning the positive part of the generalized cluster complex, and~\cite{sommers2}, concerning strictly positive ideals in the root poset).

\begin{Theorem} \label{theo_fpark}
Assume that $W$ is irreducible, with Coxeter number $h$. The representation in~\eqref{fparkbigoplus} is an $(mh-1)$-parking space, i.e., $\fpark_{W,m} = \bPsi_{mh-1}$. Equivalently,
\begin{equation} \label{eq:formula_kappaplus}
 \forall X \in \Theta,
 \quad
 \kappa^+(W,X,m) = \frac{p_X(mh-1)}{[N(W_X):W_X]}.
\end{equation}
\end{Theorem}

Theorem~\ref{theo_fpark} can be approached along the same path as Theorem~\ref{theo_park}. Rather than doing that, we will prove that Theorems~\ref{theo_park} and~\ref{theo_fpark} are equivalent. More precisely, we give two proofs of this equivalence, that follow a similar pattern:
\begin{itemize}\itemsep=0pt
 \item In Section~\ref{parkingproof1}, we show that the characters $\bPsi_{mh+1}$ and $\bPsi_{mh-1}$ satisfy a kind of inclusion-exclusion (which is straightforward for the characters $\park_{W,m}$ and $\fpark_{W,m}$).
 \item In Section~\ref{parkingproof2}, we show that the characters $\park_{W,m}$ and $\fpark_{W,m}$ satisfy $\fpark_{W,m} = (-1)^n \boldsymbol{\epsilon} \otimes \fpark_{W,-m} $ (the corresponding identity for $\bPsi_{mh+1}$ and $\bPsi_{mh-1}$ being clear).
\end{itemize}
In each case, the equivalence of the two theorems follows. These two proofs are case-free (they do not rely on the finite type classification), so it might happen that both have a role to play in a fully combinatorial and case-free approach to parking space theory.

\section[Proof of Theorem~\ref{theo_fpark} via the W-Laplacian]{Proof of Theorem~\ref{theo_fpark} via the $\boldsymbol{W}$-Laplacian}
\label{parkingproof1}

We begin this section by stating a property of the characters $\park_{W,m}$ and $\fpark_{W,m}$.

\begin{Proposition} \label{lemm:inclexcl}
We have the ``inclusion-exclusion'' formulas
\begin{align*}
 &\park_{W,m}= \sum_{I\subset S} \ind_{W_I}^W (\fpark_{W_I,m}),\\
& \fpark_{W,m}= \sum_{I\subset S} (-1)^{\#S - \#I}\ind_{W_I}^W (\park_{W_I,m}).
\end{align*}
\end{Proposition}

\begin{proof}
 First note that if $I\subset J \subset S$, we have $\ind_{W_J}^W \big( \ind_{W_I}^{W_J} ({\bf 1}) \big) = \ind_{W_I}^W({\bf 1})$. This transitivity property follows from the general theory of group characters. With this at hand, the equivalence between these identities follow from the same lines as the classical inclusion-exclusion principle.

 It remains only to prove the first identity. For $I\subset S$, let $c_I$ denote $\prod_{i\in I} s_i $ where the order is such that $c_I \sqsubset c$. We have
 \begin{align*}
 \bigoplus_{\substack{ w_1,\dots, w_m \in\NC(W,c) \\ w_1\leq \dots \leq w_m,\; \supp(w_m)=I }}
 \mathbb{C}^{ W / W_{\Fix(w_1)} }
 &=
 \bigoplus_{\substack{ w_1,\dots, w_{m-1} \in\NC(W_I,c_I), w_m \in \NC'(W_I,c_I) \\ w_1\leq \dots \leq w_m }}
 \mathbb{C}^{ W / W_{\Fix(w_1)} },
 \end{align*}
 as we can identify the indexing sets using the natural inclusion $\NC(W_I,c_I) \subset \NC(W,c)$. It is straightforward to see that $\mathbb{C}^{ W / W_{\Fix(w_1)} }$ is the induction of $\mathbb{C}^{ W_I / W_{\Fix(w_1)} }$ from $W_I$ to $W$, so that the sum in the right-hand side of the previous equation is the prime parking space of $W_I$ induced from $W_I$ to $W$. By summing the previous equation over $I$ and taking the character of the representations, we get the first identity in the proposition.
\end{proof}

Our goal is to show that the characters $\bPsi_{mh+1}$ and $\bPsi_{mh-1}$ satisfy the same relations. First note that these two characters are only defined in the irreducible case, and we need to extend them in the natural multiplicative way. For these we refer to the {\it abstract parking spaces}, as opposed to other kinds of parking spaces which are defined as characters of some explicit representation of $W$.

\begin{Definition}
If $W$ is irreducible with Coxeter number $h$, we define $\abspark_{W,m} := \bPsi_{mh+1}$ and $\absfpark_{W,m} := \bPsi_{mh-1}$. Otherwise, consider the decomposition $W = \prod_{i=1}^r W_i$ into irreducible factors. Note that there is an isomorphism
\[
 R(W) \simeq \bigotimes_{i=1}^r R(W_i)
\]
and define
\begin{align*}
 &\abspark_{W,m}:=\abspark_{W_1,m} \otimes \cdots \otimes \abspark_{W_r,m}, \\
 &\absfpark_{W,m}:= \absfpark_{W_1,m} \otimes \cdots \otimes \absfpark_{W_r,m}.
\end{align*}
More explicitly, there is a natural identification $L(W) \simeq \prod_{i=1}^r L(W_i)$. So $w=(w_1,\dots, w_r) \in W$ determines flats $Z_i \in L(W_i)$ by
\[
 \Fix(w) = \prod_{i=1}^r Z_i.
\]
We then have
\[
 \abspark_{W,m}(w) = \prod_{i=1}^r(mh_i+1)^{\dim(Z_i)},\qquad \absfpark_{W,m}(w) = \prod_{i=1}^r(mh_i-1)^{\dim(Z_i)},
\]
where $h_i$ is the Coxeter number of $W_i$.
\end{Definition}

Since the values of these characters at $w\in W$ only depend on $\Fix(w)$, we can describe their parabolic inductions in terms of the geometry of hyperplane arrangements.

\begin{Lemma} \label{lemm:evalchar}
For $X \in L(W)$, $w\in W$ and $Z=\Fix(w)$, we have
\[
 \big(\ind_{W_X}^W\absfpark_{W_X,m}\big) (w)
 =
 [N(X):W_X]\cdot\sum_{\substack{ Y \in L(W), \\ Y \sim X \text{ and } Y \subset Z }}
 \absfpark_{W_{Y},m} (w).
\]
\end{Lemma}

\begin{proof}
 By the general formula for induction of characters, we have
 \[
 \big(\ind_{W_X}^W\absfpark_{W_X,m}\big)(w)
 =
 \sum_{y\in W/W_X, \; y^{-1} w y\in W_X}
 \absfpark_{W_X,m}\big(y^{-1} w y\big).
 \]
 Note that we have
 \[
 \absfpark_{W_{(y\cdot X)},m}( w )
 =
 \absfpark_{W_X,m}\big(y^{-1} w y\big),
 \]
 as can be seen by using the inner automorphism $w\mapsto y^{-1} w y$ (and its extension to $L(W)$ which the action of $y$).
 Moreover, we have
 \[
 y^{-1} w y\in W_X
 \;\Leftrightarrow\;
 X \subset \Fix(y^{-1} w y)
 \;\Leftrightarrow\;
 X \subset y^{-1} \cdot \Fix(w)
 \;\Leftrightarrow\;
 y\cdot X \subset Z.
 \]
 We can thus rewrite the sum and get
 \[
 \big(\ind_{W_X}^W\absfpark_{W_X,m}\big)(w)
 =
 \sum_{y\in W/W_X, \; y\cdot X \subset Z }
 \absfpark_{W_{(y\cdot X)},m}( w ).
 \]
 By letting $ Y = y \cdot X$, this can be rewritten
 \begin{align*}
 \big(\ind_{W_X}^W\absfpark_{W_X,m}\big)(w)
 =
 [N(W_X):W_{X}] \cdot \sum_{\substack{ Y\in L(W) \\ Y\sim X \text{ and } Y \subset Z } } \absfpark_{W_{Y},m}(w).
 \end{align*}
 To get the equality, it suffices to check that $[N(W_{X}):W_{X}]$ is the number of $y\in W / W_X $ such that $y\cdot X = Y$, for each $Y \sim X$. This is the orbit-stabilizer theorem, as the subgroup
 $N(W_X)$ is
 \[
 N(W_X)
 =
 \{
 w\in W \colon w\cdot X = X
 \}.
 \tag*{\qed}
 \]
 \renewcommand{\qed}{}
\end{proof}

\begin{Remark} 
The previous lemma holds with $\abspark$ in place of $\absfpark$, with a completely similar proof.
\end{Remark}

Every parabolic subgroup $W_X$ of $W$ is a possibly reducible reflection group and we would like a simpler notation of the values of the parking characters $\abspark_{W_X,m}$ and $\absfpark_{W_X,m}$. For this reason, we introduce the following notation.

\begin{notation}[multisets of Coxeter numbers]
Let $W$ be an irreducible reflection group, and $X, Z \in L(W)$ such that $X\subset Z$. Assume that $W_X=W_1 \times \cdots \times W_r$ is the decomposition into irreducible factors. Write $Z = \prod_{i=1}^r Z_i$, using as above $L(W_X) \simeq \prod_{i=1}^r L(W_i)$. Note that this isomorphism is such that
\[
 \dim(Z) - \dim(X) = \sum_{i=1}^r \dim( Z_i),
\]
which can be seen by comparing the rank functions on each side. Then we write the \emph{multiset of Coxeter numbers} associated to $X$ and $Z$ as
\[
 (h_i(X,Z))_{1 \leq i \leq \dim(Z)-\dim(X)}
 :=
 ( \underbrace{h_1,\dots,h_1}_{\dim (Z_1)\text{ times}},\dots,\underbrace{h_r,\dots,h_r}_{\dim (Z_r)\text{ times}} ),
\]
where $h_i$ is the Coxeter number of $W_i$. With this notation, we can write the parking space characters for parabolic subgroups: for each $w\in W_X$ with $\Fix(w)=Z$, we have
\begin{align*}
 &\abspark_{W_X,m}(w)
=
\prod_{i=1}^{\dim (Z)-\dim (X)}(mh_i(X,Z)+1), \\
&\absfpark_{W_X,m}(w)
=
 \prod_{i=1}^{\dim(Z)-\dim(X)}(mh_i(X,Z)-1).
\end{align*}
\end{notation}

What we need is a relation involving the Coxeter numbers of $W$ and those of its parabolic subgroups that generalizes the one which appeared in~\cite[Theorem~8.8]{chapuydouvropoulos}. Recall that this was given as follows: if $W$ is irreducible with Coxeter number $h$,
\begin{equation*}
 (h+t)^n
 =
 \sum_{X\in L(W) } \Bigg(\prod_{i=1}^{n-\dim(X)} h_i(X) \Bigg) \cdot t^{\dim (X)},
\end{equation*}
where $h_i(X)$ for $1\leq i \leq n$ is the special case of $(h_i(X,Z))_{1\leq i \leq \dim(Z)-\dim(X)}$ when $Z$ is the minimal element of $L(W)$. Such a generalization (see \cite[Section~4]{douvr_recursions} for details) is the following:

\begin{Proposition}[Corollary of the restricted Laplacian recursion] \label{prop:laplacian}
For any $X,Z\in L(W)$ such that $X\subset Z$, we have the following relation between Coxeter numbers:
\begin{equation} \label{eq:corollaryoflaplacianrec}
 \prod_{i=1}^{\dim(Z)-\dim(X)} \big(h_i(X,Z)+t\big)
 =
 \sum_{X\subset Y \subset Z}
 \Bigg(\prod_{i=1}^{\dim(Z)-\dim(Y)} h_i(Y,Z) \Bigg) t^{\dim(Y)-\dim(X)}.
\end{equation}
\end{Proposition}

\begin{proof}
For $X,Z\in L(W)$ such that $X\subset Z$, consider the set
\[
 \mathcal{A}^{X,Z}
 :=
 \{
 Y/X \colon X\subset Y\subset Z,\; \dim(Y) = \dim(Z)-1
 \},
\]
which is a hyperplane arrangement in the quotient $Z/X$ \big(the essentialization of the restricted localization $\mathcal{A}_X^Z$\big). When $X$ is the $0$-dimensional subspace, we just denote $\mathcal{A}^{Z} = \mathcal{A}^{X,Z}$.

Generalizing~\cite{chapuydouvropoulos}, to this hyperplane arrangement there is an associated Laplacian $\mathcal{L}_{X,Z}$ for which each hyperplane $K\in\mathcal{A}^{X,Z}$ is taken with multiplicity equal to $h_1(K,Z)$. Its characteristic polynomial is the left-hand side of~\eqref{eq:corollaryoflaplacianrec}. This follows analogously to~\cite[Proposition~3.13]{chapuydouvropoulos}.

Now, the Laplacian recursion~\cite[Proposition~8.3]{chapuydouvropoulos} relates the characteristic polynomial of~$\mathcal{L}_{X,Z}$ to the determinants of~$\mathcal{L}_{Y,Z}$ for $X\subset Y\subset Z$, and gives precisely the above identity.
\end{proof}

The hyperplane arrangement $\mathcal{A}^{Z}$ introduced above is called a {\it restricted arrangement}, see \mbox{\cite{orliksolomon,orlikterao}}. The number \smash{$r\big(\mathcal{A}^{Z}\big)$} of regions in the complement of this arrangement is given by
\begin{equation*}
 r\big(\mathcal{A}^{Z}\big)
 =
 (-1)^{\dim(Z)} p_Z(-1)
 =
 \nu(Z) \cdot [ N(W_Z) : W_Z ].
\end{equation*}
What we need is the latter equality, see~\cite[equations~(4.1) and~(4.2)]{orliksolomon}. It can also be obtained by plugging $t=-1$ in~\eqref{psik_expansion}, using~\eqref{factornu} and~\eqref{epstensor} in the case $J=S$.

\begin{Proposition}
Assume that $W$ is irreducible, with Coxeter number $h$. For any $Z\in L(W)$, we have the following relation between Coxeter numbers:
\[
 (mh+1)^{\dim(Z)}
 =
 \sum_{X\subset Z}r\big(\mathcal{A}^X\big) \cdot \prod_{i=1}^{\dim(Z)-\dim(X)} (m h_i(X,Z) - 1).
\]
\end{Proposition}

\begin{proof}
 Consider equation~\eqref{eq:corollaryoflaplacianrec} in the case where $X$ is the $0$-dimensional subspace, so that $h_i(X,Z)=h$. By replacing $t$ with $t/m$ and multiplying both sides by $m^{\dim (Z)}$, we get
 \[
 (mh+t)^{\dim(Z)}=\sum_{Y\subset Z}\Bigg(\prod_{i=1}^{\dim(Z)-\dim(Y)} mh_i(Y,Z)\Bigg) \cdot t^{\dim(Y)}.
 \]
 By plugging $t^{\dim(Y)}=\sum_{X\subset Y}\chi\big(\mathcal{A}^X,t\big)$ (which is the inverse of~\eqref{eq:def_chi}), we obtain
\begin{align*}
 (mh+t)^{\dim(Z)}
 &=
 \sum_{X\subset Y \subset Z} \Bigg(\prod_{i=1}^{\dim (Z)-\dim (Y)}mh_i(Y,Z)\Bigg) \cdot p_{X}(t) \\
 &=
 \sum_{X \subset Z} \Bigg( \sum_{X \subset Y \subset Z } \prod_{i=1}^{\dim(Z)-\dim(Y)} mh_i(Y,Z) \Bigg) \cdot p_X(t).
\end{align*}
Again using~\eqref{eq:corollaryoflaplacianrec} (with $t=1$), this gives
\[
 (mh+t)^{\dim (Z)}=\sum_{X\subset Z}\Bigg(\prod_{i=1}^{\dim (Z)-\dim (X)}(mh_i(X,Z)+1)\Bigg)\cdot p_X(t).
\]
Now, replacing $(m,t)$ with $(-m,-1)$ and getting rid of the signs, this becomes
\[
 (mh+1)^{\dim (Z)}
 =
 \sum_{X\subset Z} \Bigg( \prod_{i=1}^{\dim (Z)-\dim (X)} (mh_i(X,Z)-1) \Bigg) \cdot
 (-1)^{\dim(X)} p_X(-1).
\]
This is precisely what we needed to prove.
\end{proof}

\begin{Theorem}
We have
\begin{equation} \label{theo_parkabsind}
 \abspark_{W,m}
 =
 \sum_{I\subset S} \ind_{W_I}^W \big( \absfpark_{W_I,m} \big).
\end{equation}
\end{Theorem}

\begin{proof}
 We assume that $W$ is irreducible, and let $h$ be its Coxeter number. We let the reader check that the reducible case follows.

 Let $\boldsymbol{\chi}$ denote the character in the right-hand side of~\eqref{theo_parkabsind}. Using the notation $\nu(X)$ (see equation~\eqref{factornu}), we have
 \[
 \boldsymbol{\chi}
 =
 \sum_{ X \in \Theta}
 \nu(X) \cdot \ind_{W_X}^W \big(\absfpark_{W_X,m}\big).
 \]
 Let $w\in W$ and $Z=\Fix(w)$. Lemma~\ref{lemm:evalchar} tells us that the evaluation of the character is
 \begin{align*}
 \boldsymbol{\chi}(w)
 =
 \sum_{X \in \Theta} \nu(X) \cdot [N(X):W_X] \cdot \sum_{\substack{ X'\in L(W), \\ X'\sim X\text{ and } X'\subset Z}}
 \absfpark_{W_{X'},m}(w).
\end{align*}
This can be rewritten as a single sum over $X'$, and we get
\begin{align*}
 \boldsymbol{\chi}(w)
 &=
 \sum_{\substack{ X\in L(W), \\ X\subset Z }} \nu(X) \cdot [N(X):W_X]\cdot
 \absfpark_{W_X,m}(w) \\
 &=
 \sum_{X\subset Z}r\big(\mathcal{A}^X\big) \cdot
 \absfpark_{W_X,m}(w).
 \end{align*}
 Indeed, recall $\nu(X) \cdot [N(X):W_X] = (-1)^{\dim(X)} p_X(-1) = r\big(\mathcal{A}^X\big)$. From Proposition~\ref{prop:laplacian}, this becomes $\boldsymbol{\chi}(w) = (mh+1)^{\dim(Z)}$, so that $\boldsymbol{\chi} = \abspark_{W,m}$.
\end{proof}

\begin{Theorem} 
We have
\[
 \absfpark_{W,m}
 =
 \sum_{I\subset S} (-1)^{\#S-\#I} \ind_{W_I}^W \big( \abspark_{W_I,m} \big).
\]
\end{Theorem}

\begin{proof}
It is straightforward to adapt the proof of the previous theorem. We can also see it as a consequence: substitute $m$ to $-m$ in~\eqref{theo_parkabsind} and tensor with $\sign$, this gives (where $\sign_{W_I}$ denotes the sign character of $W_I$ to distinguish it from $\sign$)
\begin{align*}
 \absfpark_{W,m}
 &=
 (-1)^n \sum_{I\subset S} \sign \otimes \ind_{W_I}^W \big( \absfpark_{W_I,-m} \big) \\
 &=
 (-1)^n \sum_{I\subset S} \ind_{W_I}^W \big( \sign_{W_I} \otimes \absfpark_{W_I,-m} \big) \\
 &=
 (-1)^n \sum_{I\subset S} \ind_{W_I}^W \big( (-1)^{\#I} \abspark_{W_I,m} \big).
\end{align*}
We used the identity
\begin{align} \label{eq:tensor_eps}
 \sign \otimes \ind_{W_I}^W (\boldsymbol{\chi})
 =
 \ind_{W_I}^W (\sign_{W_I} \otimes \boldsymbol{\chi}),
\end{align}
where $\boldsymbol{\chi}$ is any character of $W_I$. It easily follows from the general formula for induction, and the fact that $\sign_{W_I}$ is the restriction of $\sign$ to $W_I$. Indeed, we have for any $w\in W$
\begin{align*}
 \ind_{W_I}^W (\sign_{W_I} \otimes \boldsymbol{\chi})(w)
 &=
 \frac{1}{|W_I|}
 \sum_{x\in W, \; x^{-1} w x \in W_I}
 (\sign_{W_I} \otimes \boldsymbol{\chi}) \big(x^{-1} w x\big) \\
 &=
 \frac{ 1 }{|W_I|}
 \sum_{x\in W, \; x^{-1} w x \in W_I}
 \sign_{W_I}\big(x^{-1} w x\big) \cdot \boldsymbol{\chi} \big(x^{-1} w x\big)
\end{align*}
We have $\sign_{W_I}\big(x^{-1} w x\big) = \sign\big(x^{-1} w x\big) = \sign(w)$, since $\sign_{W_I}$ is the restriction of $\sign$ and $\sign$ is a class function. It follows
\begin{align*}
 \ind_{W_I}^W (\sign_{W_I} \otimes \boldsymbol{\chi})(w)
 &=
 \frac{ \sign(w) }{|W_I|}
 \sum_{x\in W, \; x^{-1} w x \in W_I}
 \boldsymbol{\chi} \big(x^{-1} w x\big)
 = \sign(w) \cdot \ind_{W_I}^W (\boldsymbol{\chi})(w). \qedhere
\tag*{\qed}
\end{align*}
\renewcommand{\qed}{}
\end{proof}

We have thus proved that $\abspark_{W,m}$ and $\absfpark_{W,m}$ satisfy the same identities as $\park_{W,m}$ and $\fpark_{W,m}$ in Proposition~\ref{lemm:inclexcl}.

\begin{proof}[Proof of Theorem~\ref{theo_fpark}]
We have $\park_{W_I,m} = \abspark_{W_I,m}$ for all $I\subset S$ by Theorem~\ref{theo_park}. From the second identity in Proposition~\ref{prop:inclexcl} and the previous theorem, we immediately get $\absfpark_{W,m} = \fpark_{W,m}$.
\end{proof}

\section{Proof of Theorem~\ref{theo_fpark} via combinatorial matrices}
\label{parkingproof2}

The goal of this section is to prove identity~\eqref{eq:relationchar} below, as the proof of Theorem~\ref{theo_fpark} will immediately follow.

\begin{Proposition} \label{prop:relationchar}
We have
\begin{equation} \label{eq:relationchar}
 \fpark_{W,m} = \bPsi_{-1} \otimes \; \park_{W,-m}.
\end{equation}
\end{Proposition}

\begin{proof}[Proof of Theorem~\ref{theo_fpark}]
 If $W$ is irreducible, we have $\park_{W,m} = \bPsi_{mh+1}$ by Theorem~\ref{theo_park}. The right-hand side of~\eqref{eq:relationchar} is then $\bPsi_{-1} \otimes \bPsi_{-mh+1} = \bPsi_{mh-1}$. We thus get $\fpark_{W,m} = \bPsi_{mh-1}$. The reducible case follows immediately.
\end{proof}

\begin{Remark}
 We have obtained Theorem~\ref{theo_fpark} as a consequence of Theorem~\ref{theo_park}. It is worth pointing out that Proposition~\ref{prop:relationchar} actually proves the equivalence of these two theorems.
\end{Remark}

Below, we identify elements in $R(W)$ with column vectors indexed by $\Theta$. We also use matrices with rows and columns indexed by $\Theta$. In both cases, we use bold fonts as we do for elements of $R(W)$. Note that $\otimes$ still denotes the product of $R(W)$, but we use no sign for the matrix products.

\begin{Definition}
Let $w$ be a fixed element of $\NC(W,c)$ such that $\Fix(w)\sim Y$. We define three matrices $\bQ$, $\bR$, and $\bN$ by
\begin{gather*}
 \bQ_{X,Y} = \# \big\{ v \in\NC(W,c) \colon \Fix(v) \sim X \text{ and } v\leq w \big\}, \\
 \bR_{X,Y} = \# \big\{ v \in\NC(W,c) \colon \Fix(v) \sim X \text{ and } v\ll w \big\}, \\
 \bN_{X,Y} = \# \big\{ v \in\NC(W,c) \colon \Fix(v) \sim X \text{ and } v\sqsubset w \big\}.
\end{gather*}
 Also we define $\bD$ as the diagonal matrix such that, for all $X\in \Theta$,
\[
 \bD_{X,X} = (-1)^{n - \dim X}.
\]
Moreover, $\bI$ denotes the identity matrix.
\end{Definition}

These matrices don't depend on the chosen element $w$. In the case of $\bQ_{X,Y}$, this is a consequence of the fact that the absolute order is invariant under conjugation. As for the other matrices, it's worth writing some details about the proof.

\begin{Lemma}
 The numbers $\bR_{X,Y}$ and $\bN_{X,Y}$ don't depend on the chosen element $w$.
\end{Lemma}

\begin{proof}
 Consider the sign character $\sign_{W_{\Fix(w)}} \in R\big(W_{\Fix(w)}\big)$. Via Lemma~\ref{lemm_solomon} applied to $W_{\Fix(w)}$, we have
 \[
 \sign_{W_{\Fix(w)}}
 =
 \sum_{v\sqsubset w} (-1)^{\ell(v)}\bPhi_{\Fix(v)},
 \]
 where it is understood that $\bPhi_{\Fix(v)} \in R( W_{\Fix(w)} )$. This is because the elements $v\sqsubset w$ naturally correspond to subsets of the factors in the canonical factorization of $w$, i.e., subsets of the simple system of $W_{\Fix(w)}$. By inducing on $W$ and using~\eqref{eq:tensor_eps}, we have
 \begin{align} \label{eq:tensorsignDN}
 \sign \otimes \Phi_Y = \sum_{X} (-1)^{n-\dim(X)} \bN_{X,Y} \Phi_X.
 \end{align}
 Since the left-hand side doesn't depend on the chosen $w$, the same is true for the numbers $\bN_{X,Y}$.

 The result for $\bR_{X,Y}$ follows because this number can be computed in terms of the coefficients of $\bQ$ and $\bN$:
 \begin{align} \label{eq:RQNDN}
 \bR_{X,Z} = \sum_{Y} \bQ_{X,Y} \bN_{Y,Z} (-1)^{\dim(Z)-\dim(Y)}.
 \end{align}
 This equality follows from an inclusion-exclusion procedure to compute $\bR_{X,Z}$, using Proposition~\ref{prop:order_uvw}. We omit the details.
\end{proof}

Note that the matrices $\bQ$, $\bR$, and $\bN$ are lower unitriangular (upon ordering $\Theta$ so that dimension is increasing: for example $\dim(Y)=0$ and $w=c$ gives nonzero values of $\bQ_{X,Y}$ for any $X$).

\begin{Lemma} \label{lemm:solomon2}
 For any $\bX \in R(W)$, we have $\sign \otimes \bX = \bD \bN \bX$.
\end{Lemma}

\begin{proof}
This is a rewriting of Lemma~\ref{lemm_solomon}. More explicitly, this follows from~\eqref{eq:tensorsignDN}.
\end{proof}

\begin{Lemma} \label{lemm:ninverse}
 We have $ (\bD \bN)^2 = \bI$.
\end{Lemma}

\begin{proof}
This follows from the previous proposition, since $\sign \otimes \sign = {\bf 1}$. Alternatively, one can see the relation $\bN^{-1} = \bD \bN \bD$ as a reformulation of the inclusion-exclusion principle on subsets of~$S$.
\end{proof}

\begin{Lemma} \label{lemm:RNQ}
 We have $\bR \bN = \bQ$.
\end{Lemma}

\begin{proof}
 A combinatorial expansion of the product $\bR \bN$ shows that the $X,Y$-coefficient is the number of $u$, $v$ such that $u\ll v \sqsubset w$ and $\Fix(u) \sim X$, where $w$ is a fixed element such that $\Fix(W) \sim Y$. So this is an immediate consequence of Proposition~\ref{prop:order_uvw}. (Also, equation~\eqref{eq:RQNDN} above can be rewritten $\bR = \bQ \bD \bN \bD $.)
\end{proof}

\begin{Lemma} \label{lemm:Qinverse}
 We have $\bQ^{-1} = (\bD \bN) \bQ (\bD \bN)$.
\end{Lemma}

\begin{proof}
 A combinatorial expansion of the product $ (\bR \bD)^2$ shows that the coefficient of $X,Y$ is
 \[
 \sum_{\substack{ u,v \in \NC(W,c) \\ u \ll v \ll w, \; \Fix(u) \sim X} }
 (-1)^{\ell(v) + \ell(w)},
 \]
 where $w \in \NC(W,c)$ is such that $\Fix(w) \sim Y$. By Lemma~\ref{lemm:bool}, we get an alternating sum over a Boolean lattice. So this is $0$ if $X \neq Y$, and $1$ if $X=Y$ (the two signs cancel each other out). We thus have $ (\bR \bD)^2 = \bI $.

 Since $\bR = \bQ \bN^{-1}$ by Lemma~\ref{lemm:RNQ}, we get $ \bigl(\bQ \bN^{-1} \bD\bigr)^2 = \bI $, so that
 \[
 \bQ^{-1} = \bN^{-1} \bD \bQ \bN^{-1} \bD.
 \]
 We have $ \bN^{-1} \bD = \bD \bN$ from Lemma~\ref{lemm:ninverse}, and the result follows.
\end{proof}

\begin{Definition}
 Let $\bU$ denote the column vector that corresponds to the trivial character ${\bf 1} \in R(W)$ under our convention. Explicitly, it is given by
 \[
 \bU_X = \begin{cases}
 1 &\text{if $X$ is the maximal element of $L(W)$}, \\
 0 &\text{otherwise}.
 \end{cases}
 \]
\end{Definition}

Indeed, by~\eqref{def_WX} we have $W_X = W$ if $X$ is the $0$-dimensional subspace. It follows that $\bPhi_X = \ind_{W_X}^W({\bf 1}) = {\bf 1}$.

\begin{Proposition}
We have
\begin{align}
 & \park_{W,m}= \bQ^m \bU, \qquad
 \fpark_{W,m}= \bQ^{m-1} \bR \bU. \label{park_matrices}
\end{align}
\end{Proposition}

\begin{proof}
 By a combinatorial expansion of the product $\bQ^m \bU$, we find that the coefficient of $\bPhi_X$ is the number of chains $w_1 \leq w_2 \leq \dots \leq w_m \leq c$ in $\NC(W,c)$ with $\Fix(w_1) \sim X$. Doing the same in the product $\bQ^{m-1} \bR \bU$, we find that the coefficient of $\bPhi_X$ is the number of chains $w_1 \leq w_2 \leq \dots \leq w_m \ll c$ in $\NC(W,c)$ with $\Fix(w_1) \sim X$. We thus recover the combinatorial definitions of $\park_{W,m}$ and $\fpark_{W,m}$ in~\eqref{parkdef} and~\eqref{fparkdef}.
\end{proof}

\begin{proof}[Proof of Proposition~\ref{prop:relationchar}]
By~\eqref{park_matrices} and Lemma~\ref{lemm:solomon2}, we have
\[
 \sign \otimes \park_{-m}^W
 =
 \bD \bN \bQ^{-m} \bU.
\]
We have $\bQ^{-1} = (\bD \bN) \bQ (\bD \bN)^{-1}$ by Lemma~\ref{lemm:Qinverse}, so we get
\[
 \sign \otimes \park_{-m}^W
 =
 (\bD\bN)^2 \bQ^{m} (\bD\bN)^{-1} \bU.
\]
We have $(\bD\bN)^2 = \bI$ by Lemma~\ref{lemm:ninverse}, so we get
\[
 \sign \otimes \park_{-m}^W
 =
 \bQ^m \bN^{-1} \bD \bU.
\]
We have $\bQ \bN^{-1} = \bR$ by Lemma~\ref{lemm:RNQ}, and $\bD \bU = (-1)^n \bU$ is easily checked. We thus get $ (-1)^n \bQ^{m-1} \bR \bU$ on the right-hand side of the previous equation. By the previous proposition, this is $ (-1)^n \fpark_{W,m}$. This completes the proof since $(-1)^n \sign = \bPsi_{-1}$.
\end{proof}

\section{The generalized cluster complex}
\label{sec:defclus}

We review Fomin and Reading's generalized cluster complex, introduced in~\cite{fominreading}, and Tzanaki's characterization from~\cite{tzanaki}. Afterwards, we introduce in Definition~\ref{def_facetoTheta} the natural way to associate a parabolic conjugacy class to each face of this complex. This will be the basis of the refined enumeration of faces.

Through this section, we assume that $c$ is a \emph{bipartite Coxeter element}: we can write $S= S_+ \uplus S_-$ where $S_\pm$ contains pairwise commuting reflections, and $c = c_+ c_- $ where $c_\pm$ is the product of elements in $S_\pm$.

Let $\Phi \subset V$ be a \emph{root system} for $W$ (in the sense of Coxeter groups), $\Phi_+\subset \Phi$ a set of positive roots, and $\Delta \subset \Phi_+$ the corresponding set of simple roots. For $\alpha\in \Phi$, we denote by $t_\alpha \in T$ the unique reflection such that $t(\alpha)=-\alpha$. Finally, let $\Delta_+$ and $\Delta_-$ be the sets of simple roots such that $t_\alpha \in S_\pm $ iff $\alpha \in \Delta_\pm$.

\subsection{Colored almost-positive roots}

Let $[m]=\{1,\dots,m\}$. We think of it as a set of ``colors'' that we use to define {\it colored roots}.

\begin{Definition}[{Fomin and Reading~\cite{fominreading}}]
 An {\it $m$-colored root} is a pair $(\alpha,i)$, denoted $\alpha^i$ for short, where $\alpha\in\Phi$ and $i\in [m]$. It is {\it positive} if $\alpha\in\Phi_+$. It is {\it almost-positive} if either $\alpha\in\Phi_+$, or $\alpha\in -\Delta$ and $i=1$. We denote by $\Phi^{(m)}_{\geq -1}$ the set of almost-positive $m$-colored roots, and $\Phi^{(m)}_{+}$ the set of positive $m$-colored roots.
\end{Definition}

When $m=1$, we can ignore colors and $\Phi^1_{\geq -1}$ is thus identified with $\Phi_{\geq -1} := \Phi_+ \cup (-\Delta)$. This is the set of {\it almost-positive roots} introduced by Fomin and Zelevinsky to define the cluster complex~\cite{fominzelevinsky}. In general, we also identify $\rho \in -\Delta$ with the colored root $\rho^{1}$ (as no confusion can arise).

Fomin and Zelevinsky~\cite{fominzelevinsky} defined a rotation $\mathcal{R}$ on $\Phi_{\geq -1}$ by $\mathcal{R} = \mathcal{R}_+ \circ \mathcal{R}_-$, where
\[
 \mathcal{R}_+ (\rho) =
 \begin{cases}
 \rho & \text{if } \rho \in -\Delta_-, \\
 c_+(\rho) & \text{otherwise},
 \end{cases}
 \qquad
 \mathcal{R}_- (\rho) =
 \begin{cases}
 \rho & \text{if } \rho \in -\Delta_+, \\
 c_-(\rho) & \text{otherwise}.
 \end{cases}
\]
Explicitly, we can check that
\[
 \mathcal{R}(\rho)
 =
 \begin{cases}
 -\rho &\text{if } \rho \in (-\Delta_+) \cup \Delta_-, \\
 c(\rho) &\text{otherwise}.
 \end{cases}
\]

Following Fomin and Reading~\cite{fominreading},
the higher order rotation $\mathcal{R}_m$ acting on $\Phi_{\geq -1}^{(m)}$ is defined as follows:
\begin{equation*} 
 \mathcal{R}_m\big(\alpha^i\big)
 =
 \begin{cases}
 \alpha^{i+1} & \text{if } \alpha\in\Phi_+ \text{ and } i<m, \\
 \mathcal{R}(\alpha)^1 & \text{otherwise.}
 \end{cases}
\end{equation*}
It will be useful to make the second case more explicit:
\begin{gather*} 
 \mathcal{R}_m(\alpha^m)
 =
 \begin{cases}
 (-\alpha)^1 & \text{if } \alpha\in \Delta_-, \\
 c(\alpha)^1 & \text{if } \alpha\in \Phi_+ \backslash \Delta_-,
 \end{cases}
 \qquad
 \mathcal{R}_m\big(\alpha^1\big)
 =
 \begin{cases}
 (-\alpha)^1 & \text{if } \alpha\in -\Delta_+, \\
 c(\alpha)^1 & \text{if } \alpha\in -\Delta_-.
 \end{cases}
\end{gather*}

\begin{Proposition}[\cite{fominreading}] \label{prop:mrotationorbits}
 Assume $W$ is irreducible, with Coxeter number $h$. Then each orbit $\omega$ for the action of $\mathcal{R}_m$ on $\Phi^{(m)}_{\geq -1}$ satisfies
 \begin{itemize}\itemsep=0pt\itemsep=0pt
 \item either $\# \omega = \frac{mh+2}2$ and $\# (-\Delta) \cap \omega = 1$,
 \item or $\# \omega = mh+2$ and $\# (-\Delta) \cap \omega = 2$.
 \end{itemize}
 Moreover, in the latter case the two elements $-\delta_i$ and $-\delta_j$ of $(-\Delta) \cap \omega$ are related by $w_\circ(\delta_i)=\delta_j$, where $w_\circ$ is the longest element of $W$.
\end{Proposition}

Note that the previous proposition says that negative simple roots have the same proportion in each orbit.
This is similar to a classical result by Steinberg~\cite{steinberg} for the action of the bipartite Coxeter element on $T$ by conjugation. It states that each orbit is
\begin{itemize}\itemsep=0pt
 \item either an ($h/2$)-element set containing one simple reflection,
 \item or an $h$-element set containing two simple reflections.
\end{itemize}

\subsection{The compatibility relation}

The generalized cluster complex can be defined as the flag simplicial complex associated to a~binary relation on $\Phi^{(m)}_{\geq-1}$, called {\it compatibility}.

\begin{Definition}[\cite{fominreading}] \label{def:defcompat}
 The {\it compatibility relation} $\mathrel{\|}$ is the unique symmetric and irreflexive binary relation on $\Phi^{(m)}_{\geq -1}$, characterized by the conditions:
 \begin{itemize}\itemsep=0pt
 \item for all $\alpha,\beta\in \Phi^{(m)}_{\geq -1}$, we have $\alpha \mathrel{\|} \beta$ iff $\mathcal{R}_m(\alpha) \mathrel{\|} \mathcal{R}_m(\beta)$,
 \item if $\alpha\in -\Delta$ and $\beta\in \Phi^{(m)}_{\geq -1}$, we have $\alpha \mathrel{\|} \beta$ iff $\alpha$ does not appear in the expansion of $\beta$ (forgetting its color) as a linear combination of simple roots.
 \end{itemize}
\end{Definition}

In particular, the elements of $-\Delta$ are mutually compatible by the second condition. Note that these two conditions suffice to decide if $\alpha \mathrel{\|} \beta$ holds, for any $\alpha,\beta \in \Phi^{(m)}_{\geq -1}$. Indeed, we can proceed as follows:
\begin{itemize}\itemsep=0pt
 \item by Proposition~\ref{prop:mrotationorbits}, there exists $i$ such that $\mathcal{R}_m^i(\alpha)$ is in $-\Delta$,
 \item by the first condition above, $\alpha \mathrel{\|} \beta$ holds iff $\mathcal{R}_m^i(\alpha) \mathrel{\|} \mathcal{R}_m^i(\beta)$ holds,
 \item by the second condition above, we can decide if $\mathcal{R}_m^i(\alpha) \mathrel{\|} \mathcal{R}_m^i(\beta)$ by expanding $\mathcal{R}_m^i(\beta)$ in terms of simple roots.
\end{itemize}
In particular, it follows that there exists at most one such relation $\mathrel{\|}$. Fomin and Reading~\cite{fominreading} proved that such a relation does exist.

\begin{Definition}[\cite{fominreading}]
 The {\it generalized cluster complex} $\Upsilon(W,m)$ is the flag simplicial complex generated by the compatibility relation $\|$ of Definition~\ref{def:defcompat}: its vertex set is $\Phi^{(m)}_{\geq -1}$, and its faces are sets of pairwise compatible vertices. The {\it positive part} of $\Upsilon(W,m)$, denoted $\Upsilon^+(W,m)$, is its full subcomplex having $\Phi^{(m)}_{+}$ as vertex set. In the case $m=1$, we write $\Upsilon(W) := \Upsilon(W,1)$ and $\Upsilon^+(W) := \Upsilon^+(W,1)$.
\end{Definition}

Explicit combinatorial descriptions have been given in classical types~\cite[Section~5]{fominreading}. In type~$A_{n-1}$, we can identify $\Phi^{(m)}_{\geq -1}$ with the proper diagonals of a convex $(mn+2)$-gon, and facets of $\Upsilon(W,m)$ with $(m+2)$-angulations of the same polygon. The other faces are {\it dissections} of the polygon where every inner polygon is a $(mk+2)$-gon for some $k$. The negative face (containing all negative simple roots) is given by a ``snake'' $(m+2)$-angulation, and the explicit description of $\Upsilon^+(W,m)$ can be deduced. See Fomin and Reading~\cite{fominreading} for details.

Before stating more properties of $\Upsilon(W,m)$ and $\Upsilon^+(W,m)$, let us give some reformulations of the definition.

\subsection{Reflection ordering}

An alternative characterization of the complex $\Upsilon(W,m)$ can be given via a reflection ordering. It is due to Brady and Watt~\cite{bradywatt} in the case $m=1$, and Tzanaki~\cite{tzanaki} in the general case.

\begin{Proposition}[\cite{steinberg}]
\label{prop:reforder}
There exists an indexing $\Phi_{\geq -1} = ( \alpha_i )_{1\leq i \leq nh/2+n}$ with the following properties:
\begin{itemize}\itemsep=0pt
 \item $c(\alpha_i) = \alpha_{i+n}$ for $1\leq i \leq nh/2$.
 \item $-\Delta_- = \{ \alpha_1 , \dots , \alpha_r \}$ and $\Delta_+ = \{ \alpha_{r+1} , \dots , \alpha_n \}$ $($where $r$ is the cardinality of $\Delta_-)$,
 \item $\Delta_- = \{ \alpha_{nh/2+1} , \dots , \alpha_{nh/2+r} \}$ and $-\Delta_+ = \{ \alpha_{nh/2+r+1} , \dots , \alpha_{nh/2+n} \}$.
\end{itemize}
\end{Proposition}

Note that the first two conditions above uniquely define the sequence. It remains to check the last condition, and the fact that each element in $\Phi_{\geq -1}$ appears exactly once. In fact, Steinberg's construction in~\cite{steinberg} gives an indexing of the whole root system $\Phi = ( \alpha_i )_{1\leq i \leq nh}$ from which we can extract the above indexing of $\Phi_{\geq -1}$.

\begin{Remark}
The map $\mathcal{R}$ is essentially a rotation through this indexing:
\[
 \mathcal{R}(\alpha_i) =
 \begin{cases}
 \alpha_{i+n} & \text{if } 1\leq i \leq nh/2, \\
 w_\circ (\alpha_{i-nh/2}) & \text{otherwise},
 \end{cases}
\]
where $w_\circ$ is the longest element in $W$.
\end{Remark}

The indexing of the previous proposition defines a total order $\prec$ on $\Phi_{\geq -1}$ by the condition $\alpha_i\prec\alpha_j \Longleftrightarrow i<j$. We refer to it as the {\it reflection ordering}.

\begin{Proposition}[{Brady and Watt~\cite[Section~8]{bradywatt}}]
Let $(\rho_i)_{1\leq i \leq n}$ be a tuple of $n$ distinct elements of $\Phi_{\geq -1}$, ordered so that $\rho_1 \succ \rho_2 \succ \dots \succ \rho_n$. Then it is a facet of $\Upsilon(W)$ iff $c = t_{\rho_1} t_{\rho_2} \cdots t_{\rho_n}$.
\end{Proposition}

The extension to general $m$ is as follows. Define an indexing $\Phi^{(m)}_{\geq -1} = \{ \beta_i \}_{1\leq i \leq mnh/2+n}$ (and accordingly, a total order $\prec$ on $\Phi^{(m)}_{\geq -1}$ where we omit the dependence in $m$ in the notation) as follows. As a sequence, it is obtained by the concatenation of the three following sequences:
\begin{itemize}\itemsep=0pt
 \item $\alpha^1_1,\dots,\alpha^1_r$ (elements of $-\Delta_-$ with color $1$),
 \item $ \alpha^m_{r+1}, \dots, \alpha^m_{nh/2+r},
 \alpha^{m-1}_{r+1}, \dots, \alpha^{m-1}_{nh/2+r},
 \dots, \alpha^{1}_{r+1}, \dots, \alpha^{1}_{nh/2+r}$ (colored positive roots),
 \item $\alpha^1_{nh/2+r+1},\dots,\alpha^1_{nh/2+n}$ (elements of $-\Delta_+$ with color $1$).
\end{itemize}

Note that the rotation $\mathcal{R}_m$ has {\it a priori} no simple description using this order, unlike in the case $m=1$.

\begin{Proposition}[Tzanaki~\cite{tzanaki}] \label{prop:tzanaki}
Let $(\rho_i)_{1\leq i \leq n}$ be a tuple of $n$ distinct elements of $\Phi^{(m)}_{\geq -1}$, ordered so that $\rho_1 \succ \rho_2 \succ \dots \succ \rho_n$. Then it is a facet of $\Upsilon(W,m)$ iff $c = t_{\rho_1} \cdots t_{\rho_n}$.
\end{Proposition}

In what follows, we generally assume that each face $f\in\Upsilon(W,m)$ is indexed in decreasing order, and write $f = \{ \rho_1 \succ \rho_2 \succ \cdots \succ \rho_k \}$.

This completes the definition and characterization of the generalized cluster complex. Note that our exposition is not exhaustive: another construction of $\Upsilon(W,m)$ is via {\it subword complexes}, see Stump, Thomas and Williams~\cite{stumpthomaswilliams}.

\subsection{Other properties}

Fomin and Reading proved various properties about their generalized cluster complex, some of them will be useful in the present work. First, $\Upsilon(W,m)$ and $\Upsilon^+(W,m)$ are purely $(n-1)$-dimensional. Their number of facets are respectively the Fu{\ss}--Catalan number $\Cat^{(m)}(W)$ and the positive Fu{\ss}--Catalan number $\Cat_+^{(m)}(W)$. Another property that naturally follows from the definition is:

\begin{Proposition}
 Let $\rho \in \Delta$, $s=t_\rho$, and denote the irreducible factors of $W_{(s)}$ as $W_1$, $W_2$, etc. Then the link of the vertex $-\rho$ in $\Upsilon(W,m)$ is the join $\Upsilon(W_1,m) \star \Upsilon(W_2,m) \star \cdots$.
\end{Proposition}

By the above result, it is natural to extend the definition of $\Upsilon(W,m)$ in the reducible case, by declaring that it is the join of the simplicial complexes $\Upsilon(W',m)$ where $W'$ runs through the irreducible factors of $W$. \big(In particular, this agrees with the fact that $\Cat^{(m)}(W)$ and $\Cat_+^{(m)}(W)$ are multiplicative over the irreducible factors of $W$\big).

The results of Brady and Watt and their extension by Tzanaki make clear that the cluster complex is related with minimal factorizations of the Coxeter element. In particular, to each face we can associate a noncrossing partition by taking the product of the corresponding reflections in an appropriate order (which is unique). This map has been used a lot in the literature (see for example~\cite{athanasiadistzanaki1}), but let us describe it explicitly.

\begin{Definition}
 Let $f = \{ \rho_1 \succ \dots \succ \rho_k \} \in \Upsilon(W,m)$. We define
 \[
 \prod f := t_{\rho_1} \cdots t_{\rho_k} \in \NC(W,c).
 \]
\end{Definition}

To see that this is well-defined, first note that from Proposition~\ref{prop:tzanaki} we have $t_{\rho_1} \cdots t_{\rho_k} \allowbreak = c$ if~$k=n$ (i.e., $f$ is a facet). In general, since a face is a subset of a facet we get that $t_{\rho_1} \cdots t_{\rho_k}$ is a~subword of a minimal reflection factorization of $c$. It follows that $\prod f \in \NC(W,c)$.

\begin{Proposition} \label{lemm:upsilonfusscat}
 For $w\in \NC(W,c)$, we have
 \[
 \# \Big\{ f \in \Upsilon^+(W,m) \colon \prod f = w \Big\}
 =
 \Cat^{(m)}_+\big(W_{\Fix(w)}\big).
 \]
\end{Proposition}

\begin{proof}
 Essentially, it is possible to identify those $f \in \Upsilon^+(W,m)$ such that $\prod f = w$ with facets of $\Upsilon^+\big(W_{\Fix(w)},m\big)$. However, some care is needed to do that. Although $w$ is a Coxeter element of $W_{\Fix(w)}$ by Proposition~\ref{prop:wcox}, it might not be bipartite Coxeter element. We thus need to use the generalized cluster complex associated to any standard Coxeter element, following the definition of Stump, Thomas, Williams~\cite{stumpthomaswilliams}.

 An alternative path is to first get the result for $m=1$. Indeed,~\cite[Proposition~6.3]{bianejosuat} gives a characterization of the compatibility relation on (uncolored) positive roots which is clearly stable by restriction to parabolic subgroups, allowing the identification mentioned above. Now, we have
 \begin{equation*} 
 \# \Big\{ f \in \Upsilon^+(W,m) \colon \prod f = w \Big\}
 =
 \sum_{w = w_1 \cdots w_m }
 \prod_{i=1}^m
 \# \Big\{ f \in \Upsilon^+(W) \colon \prod f = w_i \Big\},
 \end{equation*}
 where we sum over length-additive factorizations with $m$ factors. Indeed, this can be proved bijectively: to $f \in \Upsilon^+(W,m)$ we associate $f_1,\dots,f_m \in \Upsilon^+(W)$ where $f_i$ contains roots in $f$ of color $i$. Using the case $m=1$ of the proposition, we get
 \[
 \# \Big\{ f \in \Upsilon^+(W,m) \colon \prod f = w \Big\}
 =
 \sum_{w = w_1 \cdots w_m }
 \prod_{i=1}^m \Cat_+\big(W_{\Fix(w_i)}\big).
 \]
 The result follows from the identity
 \[
 \Cat^{(m)}_+\big(W_{\Fix(w)}\big)
 =
 \sum_{w = w_1 \cdots w_m } \prod_{i=1}^m \Cat_+\big(W_{\Fix(w_i)}\big),
 \]
 where we sum over length-additive factorizations with $m$ factors. This identity can be proved as follows: since it is invariant under conjugating $w$, we can assume $w$ is a bipartite Coxeter element of $W_{\Fix(w)}$, and the result comes from double counting of facets in $\Upsilon(W_{\Fix(w)},m)$ as above.
\end{proof}

\subsection{The map to orbits of flats}

The {\it Kreweras complement} is an anti-automorphism of $\NC(W,c)$ defined by $K(w) := c w^{-1}$. We use here a variant adapted to the bipartite Coxeter element.

\begin{Lemma} 
 The map $w\mapsto c_+ w c_-$ is an involutive anti-automorphism of $\NC(W,c)$.
\end{Lemma}

See, for example,~\cite[Proposition~3.5]{bianejosuat}.

\begin{Definition}
 \label{def_facetoTheta}
 For $f \in \Upsilon(W,m)$, we define
 \[
 \underline{f} := c_+ \Big( \prod f \Big) c_-,
 \]
 which is in $\NC(W,c)$ by the previous lemma. The natural map to orbits of flats is
 \begin{align*}
 \Upsilon(W,m) & \to L(W)/W, \\
 f &\mapsto \text{ the orbit of } \Fix(\underline{f}).
 \end{align*}
 Accordingly, we define for any $X\in\Theta$
 \begin{align*}
 &\Upsilon(W, X, m)
:=
 \big\{
 f\in \Upsilon(W, m)
 \colon \Fix(\underline{f}) \sim X
 \big\}, \\
 &\Upsilon^+(W, X, m)
:=
 \big\{
 f\in \Upsilon^+(W, m)
 \colon \Fix(\underline{f}) \sim X
 \big\},
 \end{align*}
 and their cardinalities
 \begin{align*}
 &\gamma(W, X, m )
:=
 \# \Upsilon(W, X, m), \\
 &\gamma^+(W, X, m )
:=
 \# \Upsilon^+(W, X, m).
 \end{align*}
\end{Definition}

\begin{Remark}
 We could actually take the usual Kreweras complement $K(w)$ instead of its bipartite variant. Indeed, $c w^{-1}$ and $c_+ w c_-$ are conjugate (it easily follows from the fact that each element is conjugate to its inverse, see~\cite[Corollary~3.2.14]{geckpfeiffer}). We found the bipartite Kreweras complement convenient to do some calculations, but of course everything could be done with the usual one.
\end{Remark}

Let $(f_i)_{-1\leq i \leq n-1}$ denote the $f$-vector of $\Upsilon(W,m)$, i.e., $f_i$ is the number of $i$-dimensional faces in the complex. Similarly, $(f^+_i)_{-1\leq i \leq n-1}$ denote the $f$-vector of $\Upsilon^+(W,m)$. In terms of $\gamma$ and~$\gamma^+$, we thus have:
\begin{align} \label{eq:formula_fk}
 f_{k-1} = \sum_{X\in \Theta, \; \dim(X)=k} \gamma(W, X, m)
 \qquad \text{and} \qquad
 f^+_{k-1} = \sum_{X\in \Theta, \; \dim(X)=k} \gamma^+(W, X, m).
\end{align}

In the next sections, we give the explicit formulas for the quantities $\gamma$ and $\gamma^+$. Another important property is that the natural map $\Upsilon(W,m) \to L(W)/W$ in the previous definition is invariant under the rotation $\mathcal{R}_m$ (see Proposition~\ref{prop:invariance}).

Observe that $\gamma(W, X, m )$ and $ \gamma^+(W, X, m )$ are polynomial in $m$. Indeed, the quantity in Proposition~\ref{lemm:upsilonfusscat} is polynomial (from the definition of Fu{\ss}--Catalan numbers), and by summing over a finite set of $w$ we get $\gamma(W, X, m )$ or $ \gamma^+(W, X, m )$.

\begin{Remark} 
 As a motivation for the previous definition, let us give the following statement. If $f$ is a face of $\Upsilon(W,m)$, its {\it link} is the simplicial complex
 \[
 \operatorname{Link}(f)
 :=
 \big\{ f' \in \Upsilon(W,m) \colon f\cap f' = \varnothing, \; f\cup f' \in \Upsilon(W,m)
 \big\}.
 \]
 Then $\operatorname{Link}(f)$ is isomorphic to $\Upsilon(W_{\bar f},m)$. The proof can be sketched as follows. By invariance under the rotation $\mathcal{R}_m$ (using Proposition~\ref{prop:invariance} below), we can assume that $f$ contains a negative simple root $-\alpha$. Then, we can identify $\operatorname{Link}(f)$ with the link of $f\backslash\{-\alpha\}$ in $\Upsilon(W_I,m)$ (where $I = S\backslash\{t_{\alpha}\}$), and we use induction on the rank of $W$.
\end{Remark}

For example, consider type $A_{n-1}$, where a face $f \in \Upsilon( \mathfrak{S}_n ,m)$ is identified with a dissection of a convex $(mn+2)$-gon. The link of $f$ is a join of generalized cluster complexes, where each inner $(mk+2)$-gon in the dissection contributes to a factor $\Upsilon( \mathfrak{S}_k ,m)$. Accordingly, the orbit of $\Fix(\underline{f})$ is an integer partition, where each inner $(mk+2)$-gon in the dissection contributes to a~part~$k$.

\section{Combinatorial reciprocities}
\label{sec:reciprocities}

In this section, we prove combinatorial reciprocities between the quantities $\kappa$, $\kappa^+$ on one side and $\gamma^+$, $\gamma$ on the other side (Theorem~\ref{reciprocity} below). We deduce the formulas for $\gamma^+$, $\gamma$ in terms of characteristic polynomials (or Orlik--Solomon exponents) in Corollary~\ref{formulas_gammagammaplus}, as a consequence of the combinatorial reciprocities together with the formulas for $\kappa$ and $\kappa^+$ in Theorems~\ref{theo_park} and~\ref{theo_fpark}.

\begin{Theorem} \label{reciprocity}
 We have
 \begin{align}
 \label{eq:reciprocity1}
 &(-1)^k \kappa(W, X, -m)
=
 \gamma^+(W, X, m), \\
 \label{eq:reciprocity2}
 &(-1)^k \kappa^+(W, X, -m)
=
 \gamma(W, X, m).
 \end{align}
\end{Theorem}

Before proving this, we need to state inclusion-exclusion formulas that relate $\kappa(W, X, m)$ to~$\kappa^+(W, X, m)$ on one side, and $\gamma(W, X, m)$ to $\gamma^+(W, X,m)$ on the other side. To do this, it is convenient to extend the definitions of these quantities. Using the bijection from Lemma~\ref{lemm:bij_xi}, we write $\kappa(W, \mathcal{X} , m )$ in place of $\kappa(W, X , m )$ and similarly for $\kappa^+$, $\gamma$ and $\gamma^+$. Finally, we extend this definition in an additive way: if $\mathcal{Y} \subset W$ is a disjoint union of parabolic conjugacy classes, say $\mathcal{Y} = \biguplus_{i=1}^j \mathcal{X}_i$, we write
\[
 \kappa(W, \mathcal{Y} , m )
 =
 \sum_{i=1}^j \kappa(W, \mathcal{X}_i , m )
\]
and similarly for $\kappa^+$, $\gamma$, and $\gamma^+$. In particular, for each standard parabolic subgroup $W_I \subset W$ and $\mathcal{X}$ a parabolic conjugacy class of $W$, the intersection $\mathcal{X} \cap W_I$ is such a union.

Now, we can state:

\begin{Proposition} \label{prop:inclexcl}
We have
\begin{align}
 &\kappa(W, \mathcal{X} , m )
=
 \sum_{I \subset S}
 \kappa^+(W_I, \mathcal{X} \cap W_I , m ), \nonumber\\
 \label{IE2}
 &\kappa^+(W, \mathcal{X} , m )
=
 \sum_{I \subset S} (-1)^{n-\#I}
 \kappa(W_I, \mathcal{X} \cap W_I , m ),
\end{align}
and
\begin{align}
 \label{IE3}
 &\gamma( W, \mathcal{X} , m )
=
 \sum_{I \subset S}
 \gamma^+(W_I, \mathcal{X} \cap W_I , m ),\\
 &\gamma^+(W, \mathcal{X} , m )
=
 \sum_{I \subset S} (-1)^{n-\#I}
 \gamma(W_I, \mathcal{X} \cap W_I , m ). \nonumber
\end{align}
\end{Proposition}

\begin{proof}
The number $\kappa^+(W_I, \mathcal{X} \cap W_I , m )$ counts chains $w_1 \leq \dots \leq w_m$ in $\NC(W_I,c_I)$ where $w_1 \in \mathcal{X} \cap W_I $ and $w_m$ has full support in $W_I$. By summing over $I$, we get the first equation. The second can be deduced, essentially via the same proof as in the classical inclusion-exclusion principle.

Let's fix $I\subset S$. Then the map
\[
 f \mapsto f^+ := f \cap \Phi^{(m)}_+
\]
is a bijection from faces $f\in \Upsilon(W,m)$ such that $f\cap (-\Delta) = \{ -\delta_i \colon i\in S\backslash I \}$ to $\Upsilon^+(W_I,m)$. To check the relation between $\underline{f}$ and $\underline{f^+}$, note that with $f$ as above we have
\[
 \prod f
 =
 \bigg( \prod_{i\in S^+ \backslash I } s_i \bigg)
 \Big( \prod f^+ \Big)
 \bigg( \prod_{i\in S^- \backslash I } s_i \bigg),
\]
so that
\[
 \underline{f}
 =
 \bigg( \prod_{i\in S^+ \cap I } s_i \bigg)
 \Big( \prod f^+ \Big)
 \bigg( \prod_{i\in S^- \cap I } s_i \bigg) = \underline{f^+}.
\]
The last equality comes from the fact that $\underline{f^+}$ is defined with respect to the bipartite Coxeter element of $W_I$. It follows that the condition $\underline{f} \in \mathcal{X}$ is equivalent to $\underline{f^+} \in \mathcal{X}\cap W_I$. By summing over $I\subset S$, we get~\eqref{IE3}. The inverse relation follows, via inclusion-exclusion as before.
\end{proof}

\begin{proof}[Proof of Theorem~\ref{reciprocity}]
Recall that $\kappa(W, \mathcal{X}, m)$ counts chains $w_1 \leq \dots \leq w_m$ in $\NC(W,c)$ where $w_1 \in \mathcal{X} $. If $w_1$ is fixed, the other elements $w_2,\dots,w_m$ can be mapped to $w_1^{-1} w_2, \dots ,\allowbreak w_1^{-1} w_m$ and are thus in bijection with $(m-1)$-element chains in $\NC(W,c)$ where the top element is below $w_1^{-1}c$. The number of such tuples $(w_2,\dots,w_m)$ is thus given by the Fu{\ss}--Catalan number $\Cat^{(m-1)}\big( W_{\Fix(w_1^{-1}c)} \big)$. We get
\begin{equation} \label{eq:kappacat}
 \kappa(W, \mathcal{X}, m)
 =
 \sum_{w \in \mathcal{X} } \Cat^{(m-1)}\big( W_{\Fix(w^{-1}c)} \big).
\end{equation}
Note that we have a combinatorial reciprocity for Fu{\ss}--Catalan numbers
\begin{align*} 
 \Cat^{(-m)}(W) = (-1)^n \Cat^{(m-1)}_+(W).
\end{align*}
This relation follows from the two formulas in terms of the exponents given in the introduction, and the fact that the (increasingly sorted) exponents satisfy $e_i = h-e_{n+1-i}$. Using this reciprocity, equation~\eqref{eq:kappacat} gives
\[
 (-1)^k \kappa(W,\mathcal{X},-m)
 =
 \sum_{w \in \mathcal{X} } (-1)^k \Cat^{(-m-1)}\big( W_{\Fix(w^{-1}c)} \big)
 =
 \sum_{w \in \mathcal{X} } \Cat_+^{(m)}\big( W_{\Fix(w^{-1}c)} \big).
\]
Since $w^{-1}c$ is conjugate to $c_+ w c_-$, we also have
\[
 (-1)^k \kappa(W,\mathcal{X},-m)
 =
 \sum_{w \in \mathcal{X} } \Cat_+^{(m)}\big( W_{\Fix(c_+ w c_-)} \big).
\]
By Proposition~\ref{lemm:upsilonfusscat}, each term $\Cat_+^{(m)}\big( W_{\Fix(c_+ w c_-)} \big)$ is the number of positive faces in $f \in \Upsilon^+(W,m)$ such that $\prod f = c_+ w c_-$. This is also the number of positive faces $f \in \Upsilon^+(W,m)$ such that $\underline{f} = w$. So the sum is $\gamma^+(W, \mathcal{X}, m)$ by definition, and we have proved~\eqref{eq:reciprocity1}.

Then, by substitution $m\leftarrow -m$ in~\eqref{IE2} we get
\begin{align} \label{eq:kappasumkappa}
 \kappa^+( W, \mathcal{X} , -m )
 =
 \sum_{I \subset S} (-1)^{n-\#I}
 \kappa( W_I, \mathcal{X} \cap W_I , -m ).
\end{align}
Using~\eqref{eq:reciprocity1} that we have just proved, we get
\[
 \kappa( W_I , \mathcal{X} \cap W_I, -m )
 =
 (-1)^{\#I - (n-k) } \gamma^+( W_I , \mathcal{X} \cap W_I , m ).
\]
To check the sign, note that in Theorem~\ref{reciprocity} the integer $k = \dim(X)$ is $n-\ell(w)$ for some $w\in\mathcal{X}$. Here, the elements of $\mathcal{X} \cap W_I$ have reflection length $n-k$ in $W_I$ (it is easily seen to be the same as their reflection length in $W$). In $W_I$, which has rank $\#I$, this sign is thus given by the difference $\#I - (n-k)$. Via the previous equation,~\eqref{eq:kappasumkappa} becomes
\[
 \kappa^+(W, \mathcal{X}, -m )
 =
 (-1)^k \sum_{I \subset S}
 \gamma^+(W_I, \mathcal{X}\cap W_I , m ).
\]
From~\eqref{IE3}, the right-hand side of the previous equation is $(-1)^k \gamma(W, \mathcal{X}, m)$. We have thus proved~\eqref{eq:reciprocity2}.
\end{proof}

\begin{Corollary} \label{formulas_gammagammaplus}
 We have
 \begin{align}
 &\gamma(W,X,m)
 =
 (-1)^{\dim(X)} \frac{p_X(-mh-1)}{[N(W_X):W_X]},\label{eq:formulagamma}\\
 &\gamma^+(W,X,m)
 =
 (-1)^{\dim(X)} \frac{p_X(-mh+1)}{[N(W_X):W_X]}. \nonumber
 \end{align}
\end{Corollary}

\begin{proof}
 Using the combinatorial reciprocities in Theorem~\ref{reciprocity}, this follows from the corresponding formulas for $\kappa$ and $\kappa^+$ in equations~\eqref{eq:formula_kappa} and~\eqref{eq:formula_kappaplus}.
\end{proof}

\section[Bijections between faces of the generalized cluster complex and chains of noncrossing partitions]{Bijections between faces of the generalized cluster complex \\ and chains of noncrossing partitions}
\label{sec:bij}

We give another proof of the formulas in Corollary~\ref{formulas_gammagammaplus}, via a bijection which is of independent interest. This is a bijection between faces $f \in \Upsilon(W,X,m)$ and chains of noncrossing partitions of the form $
w_0 \sqsubset w_1 \leq \dots \leq w_m$ such that $\Fix(w_0) \sim X$, and it will give a proof of Proposition~\ref{prop:bijgammachi}.

\begin{Remark}
The combinatorial reciprocities in the previous section gave connections between~$\kappa$ and~$\gamma^+$ on one side (equation~\eqref{eq:reciprocity1}), and between $\kappa^+$ and $\gamma$ on the other side (equation~\eqref{eq:reciprocity2}). Here the bijection will give connections between $\kappa$ and $\gamma$ on one side (equation~\eqref{EQ:gammakappa}), and between $\kappa^+ $ and $\gamma^+$ on the other side (equation~\eqref{EQ:gammakappa2}).
\end{Remark}

\begin{Proposition} \label{prop:bijgammachi}
We have
\begin{align}
 &\label{EQ:gammakappa}
 \gamma(W,X,m)
=
 \sum_{Y \in \Theta} \bN_{X,Y} \kappa(W,Y,m), \\
 &\label{EQ:gammakappa2}
 \gamma^+(W,X,m)
=
 \sum_{Y \in \Theta} \bN_{X,Y} \kappa^+(W,Y,m).
\end{align}
\end{Proposition}

Note that these two identities are related to each other: one is the consequence of the other, using the combinatorial reciprocities in~\eqref{eq:reciprocity1} and~\eqref{eq:reciprocity2}, together with the relation $\bN^{-1} = \bD \bN \bD$. Here we prove both identities: the bijection proving the first one, suitably restricted, also proves the second one.

One of our main tools is the following (see also Lemma~\ref{prop:order_uvw} for a related result).

\begin{Lemma} \label{lemm:intervalsfaces}
For any $u,w \in\NC(W,c)$ such that $u\leq w$, the elements $v\in\NC(W,c)$ such that $u\sqsubset v \ll w$ are in bijection with faces $f \in \Upsilon^+(W)$ such that $\prod f = u^{-1} w$.
\end{Lemma}

\begin{proof}
 This follows from the results in~\cite{bianejosuat} (though in this reference we only deal with the case where $u$ is the minimal element, this slight generalization is proved similarly). More explicitly, the construction is as follows.

 Start from $f=\{\alpha_1,\dots, \alpha_k\}$ as above, and let $t_i = t_{\alpha_i}$. By~\cite[Lemmas~8.7 and 8.8]{bianejosuat}, we can reindex the elements of $f$ (switching pairs of orthogonal reflections) so that $u \sqsubset ut_1 \sqsubset \null \cdots \null \sqsubset ut_1 \cdots t_j \ll u t_1 \cdots t_{j+1} \ll \null\cdots\null \ll ut_1 \cdots t_k = w$. Then the bijection sends $f$ to $ut_1 \cdots t_j$.

 In the other direction, write $u^{-1} v = t_1 \cdots t_k$ where the factors are the simple reflections of $W_{\Fix(u^{-1}v)}$, and similarly $v^{-1} w = u_1 \cdots u_j$ where the factors are the simple reflections of $W_{\Fix(v^{-1}w)}$. Then the reflections $t_\alpha$ for $\alpha\in f$ are
 \begin{itemize}\itemsep=0pt
 \item $t_k \cdots t_{i+1} t_i t_{i+1} \cdots t_k$ for $1\leq i \leq k$,
 \item $u_1 \cdots u_{i-1} u_i u_{i-1} \cdots u_1$ for $1\leq i \leq j$.
 \end{itemize}
See~\cite[Section~8]{bianejosuat} for details.
\end{proof}

\begin{proof}[Proof of Proposition~\ref{prop:bijgammachi}]
We define the bijection announced at the beginning of this section, between faces $f \in \Upsilon(W,X,m)$ and chains of noncrossing partitions of the form $
w_0 \sqsubset w_1 \leq \dots \leq w_m$ such that $\Fix(w_0) \sim X$.

Let us start from a chain in $\NC(W,c)$ as above. Using Proposition~\ref{prop:order_uvw}, a chain $w_0 \sqsubset w_1 \leq \dots \leq w_m \leq c $ can be completed in a unique way as a chain
\[
 w_0 \sqsubset w_1 \ll w'_1 \sqsubset w_2 \ll w'_2 \sqsubset \dots w_{m-1} \ll w'_{m-1} \sqsubset w_m \ll w'_m \sqsubset c.
\]
Let us denote $w'_0 := w_0$ for convenience. Each element $w_i$ (with $1\leq i \leq m$) is such that $w'_{i-1} \sqsubset w_i \ll w'_i$. Via Lemma~\ref{lemm:intervalsfaces}, it corresponds bijectively to a face $f_i \in \Upsilon^+(W)$ such that
\begin{equation} \label{eq:prod_ra}
 \prod f_i = (w'_{i-1})^{-1} w'_i.
\end{equation}
Define $f\in \Upsilon(W,m)$ as follows:
\begin{itemize}\itemsep=0pt
 \item $f$ contains the colored positive roots $\alpha^i$ for $\alpha\in f_i$,
 \item $f$ contains the negative roots $-\delta_i$ for $i \in S\backslash \supp(w'_m)$.
\end{itemize}
To check that this is indeed in $\Upsilon(W,m)$, first note that from~\eqref{eq:prod_ra} we get
\begin{equation} \label{prod_ra2}
 \prod f^+ = (w'_{0})^{-1} w'_m
\end{equation}
with the notation $f^+ = f \cap \Phi^{(m)}_+$ as before. Indeed, the order is such that (reading this product from left to right) we first read roots with color $1$, then roots with color $2$, etc. By Proposition~\ref{prop:tzanaki}, this shows that $f^+ \in \Upsilon^+(W,m)$. Then, each element $-\delta_i \in f $ is such that $s_i \notin \supp(w_m)$. Consequently, $s_i \notin \supp(t)$ for each $t\in T$ such that $t\leq w_m$. Thus this $-\delta_i$ is compatible with the roots in $f^+$, and it follows $f\in \Upsilon(W,m)$.

To compute $\underline{f}$, let's keep the notation as in the proof of Proposition~\ref{prop:inclexcl} so that $f \cap (-\Delta) = \{ -\delta_i \colon i\in S\backslash I \}$ for some $I\subset S$. From~\eqref{prod_ra2}, we get
\[
 \prod f =
 \bigg( \prod_{i\in S^+ \backslash I} \bigg )
 (w'_{0})^{-1} w'_m
 \bigg( \prod_{i\in S^- \backslash I} \bigg)
\]
so that
\[
 \underline{f}
 =
 \bigg( \prod_{i\in S^+ \cap I} \bigg )
 (w'_{0})^{-1} w'_m
 \bigg( \prod_{i\in S^- \cap I} \bigg ).
\]
By definition, we have
\[
 w'_m =
 \bigg( \prod_{i\in S^+\cap I} \bigg)
 \bigg( \prod_{i\in S^-\cap I} \bigg).
\]
Indeed, the condition $w'_m \sqsubset c$ implies that a reduced factorization of $w'_m$ can be extracted from the canonical one of $c = c_+ c_-$, and the definition of $f$ implies that the elements that appear are those indexed by $I$. From the previous two equations, we get
\[
 \underline{f} =
 \bigg( \prod_{i\in S^+\cap I} \bigg)
 (w'_{0})^{-1}
 \bigg( \prod_{i\in S^+\cap I} \bigg).
\]
We thus have $f\in \Upsilon(W,X,m)$, since $\Fix(\underline{f}) \sim \Fix( (w'_0)^{-1} ) \sim \Fix( w'_0 ) \sim X$.

Describing the inverse bijection is straightforward. Starting from $f\in \Upsilon(W,m)$, define
\begin{itemize}\itemsep=0pt
 \item $w'_m \sqsubset c$ is the product of the reflections $t_\alpha$ for $\alpha \in (-\Delta) \backslash f$,
 \item $ (w'_{i-1})^{-1} w'_i$ is the product of $t_{\alpha^i}$ over positive roots $\alpha^i$ (of color $i$) in $f$.
\end{itemize}
From $w'_{i-1}$, $w'_i$, and the positive roots of color $i$ in $f$, we use the inverse bijection from Lemma~\ref{lemm:intervalsfaces} to get $w_i$ such that $w'_{i-1} \sqsubset w_i \ll w'_i$. We omit details about checking that the two maps are indeed inverse bijections.

We thus have proved~\eqref{EQ:gammakappa}. Finally, observe that when we restrict this bijection to chains such that $w_m \ll c$, we get $w'_m = c$ (keeping the same notation), and the corresponding faces $f \in \Upsilon(W,\mathcal{X},m)$ contain no vertex in $-\Delta$. So the bijection proves~\eqref{EQ:gammakappa2} as well.
\end{proof}

In the case where $X$ is the minimal flat, our bijection specializes into a bijection between facets of $\Upsilon(W,m)$ and chains $w_1\leq\dots \leq w_m$. Such a bijection was first given via representation theory by Buan, Reiten, and Thomas~\cite{buanreitenthomas}. A more combinatorial one and various related bijections were given by Stump, Thomas, and Williams~\cite{stumpthomaswilliams}.

To end this section, we use the previous bijection to get the proof of Corollary~\ref{formulas_gammagammaplus}.

\begin{proof}[Proof of Corollary~\ref{formulas_gammagammaplus}]
From Lemma~\ref{lemm:solomon2} and $\bPsi_{-1} = (-1)^n \sign$, we have
\[
(-1)^n \bD \bPsi_{-t} = \bN \bPsi_{t}.
\]
By taking the coefficients via Proposition~\ref{prop:psik_expansion} and evaluating at $t=mh+1$ and $t=mh-1$, this respectively gives
\begin{align*} 
 &(-1)^{\dim(X)} \frac{p_X(-mh-1)}{[N(W_X):W_X]}
=
 \sum_{Y \in \Theta} \bN_{X,Y}
 \frac{p_Y(mh+1)}{[N(W_Y):W_Y]}, \\
 &(-1)^{\dim(X)} \frac{p_X(-mh+1)}{[N(W_X):W_X]}
=
 \sum_{Y \in \Theta} \bN_{X,Y}
 \frac{p_Y(mh-1)}{[N(W_Y):W_Y]}.
\end{align*}
Using the formulas for $\kappa$ and $\kappa^+$ obtained in~\eqref{eq:formula_kappa} and~\eqref{eq:formula_kappaplus}, we see that the right-hand sides of the previous two equations give, respectively, the right-hand sides in~\eqref{EQ:gammakappa} and~\eqref{EQ:gammakappa2}.

By identifying the left-hand sides of the previous two equations with those of~\eqref{EQ:gammakappa} and~\eqref{EQ:gammakappa2}, we immediately obtain the formulas for $\gamma$ and $\gamma^+$ in Corollary~\ref{formulas_gammagammaplus}.
\end{proof}

\section[Numerology of f- and h-vectors]{Numerology of $\boldsymbol{f}$- and $\boldsymbol{h}$-vectors}
\label{sec:fhvectors}

Recall that the integers $f_k$ are the entries of the $f$-vector of $\Upsilon(W,m)$. First, let us explain how our formulas for $\gamma$ completely explain some partial factorizations of $f_k$ as a polynomial in~$m$, obtained by Fomin and Reading~\cite[Section~8]{fominreading}. They introduced the notion of {\it level} of an exponent $e_i$, by the rule that $mh+1+e_i$ is a factor of $f_k$ (as a polynomial in $m$) iff $e_i$ has level at most $k$.

By examining the tables in \cite[Appendix~C]{orlikterao}, we note that the Orlik--Solomon exponents $b_i^X$ of $X\in \Theta$ (with $\dim(X)=k)$ often look like $e_1,\dots,e_k$, namely the $k$ smallest exponents of $W$ (in fact, this happens exactly for the so-called {\it coincidental types}~\cite[Section~3.1.5]{williams_thesis}:~$A_n$,~$B_n$,~$I_2(m)$, and~$H_3$, see \cite[Section~3.3]{miller_foulkes}). Now, consider the formula for $f_{k-1}$ in~\eqref{eq:formula_fk} as a sum of polynomials in $m$. If some integer $b$ is an Orlik--Solomon exponent for all $X \in \Theta$ with $\dim(X) = k$, each term contains a factor $(mh+1+b)$ so that their sum $f_{k-1}$ also has this factor. Also, we can observe that there is an equivalence between the two statements:
\begin{itemize}\itemsep=0pt
 \item The integer $b$ is an Orlik--Solomon exponent for all $X\in\Theta$ with $\dim(X) = k$.
 \item The integer $b$ is an Orlik--Solomon exponent for all $X\in\Theta$ with $\dim(X) \geq k$.
\end{itemize}
(It is easily checked on a case-by-case basis.) Note that it might {\it a priori} happen that $f_k$ contains a factor $mh+1-b$ even though some of the terms in the sum of~\eqref{eq:formula_fk} don't contain it. This actually never happen. The upshot of this discussion is:

\begin{Proposition}
An exponent $e_i$ of $W$ has level $j$ in the sense of Fomin and Reading~{\rm \cite{fominreading}} if and only if $e_i$ is an Orlik--Solomon exponent of all $X\in \Theta$ with $\dim(X) \geq j$.
\end{Proposition}

For example in type $F_4$, the exponents $1$, $5$, $7$, $11$ have respective levels $1$, $3$, $3$ and $4$ (see~\cite[Table~1]{fominreading}). The parabolic conjugacy class of type $B_2$ has Orlik--Solomon exponents $1$ and $3$, which differ from the two smallest exponents $1$ and $5$ (and this is the sole discrepancy, see~\cite[Table~C.9]{orlikterao}). This Orlik--Solomon exponent $3$ (rather than $5$) explains that the exponent $5$ has level $3$ (rather than $2$).

Now, consider the $h$-vector $(h_i)_{0\leq i \leq n}$ of $\Upsilon(W,m)$ and $\big(h^+_i\big)_{0\leq i \leq n}$ of $\Upsilon^+(W,m)$, defined in terms of the $f$-vector via the polynomial relation
\begin{align} \label{eq:fh-transform}
 \sum_{i=0}^n f_{i-1} z^{n-i}
 =
 \sum_{i=0}^n h_i (z+1)^{n-i},
 \qquad
 \sum_{i=0}^n f^+_{i-1} z^{n-i}
 =
 \sum_{i=0}^n h^+_i (z+1)^{n-i}.
\end{align}
These $h$-vectors have been obtained by Athanasiadis and Tzanaki~\cite{athanasiadistzanaki2,athanasiadistzanaki1} and coincide with the Fu{\ss}--Narayana numbers of Armstrong~\cite[Chapter~5]{armstrong}. These results can be stated as follows:

\begin{Proposition}[Athanasiadis and Tzanaki~\cite{athanasiadistzanaki2,athanasiadistzanaki1}] \label{prop:hvectors}
Recalling the definition of Narayana numbers in~\eqref{def:narayana} and~\eqref{def:narayanap}, for $0\leq k\leq n$ we have
\[
 h_{n-k}= \operatorname{Nar}^{(m)}(W,k), \qquad
 h^+_{n-k}= \operatorname{Nar}_+^{(m)}(W,k).
\]
\end{Proposition}

Let us explain how the formulas in equations~\eqref{EQ:gammakappa} and~\eqref{EQ:gammakappa2} can be seen as a refinement of the $f$- to $h$-vector transformation. By summing~\eqref{EQ:gammakappa} over $X\in\Theta$ of a given dimension $k$, the left-hand side gives $f_{k-1}$ (see equation~\eqref{eq:formula_fk}). The sum in the right-hand side is
\begin{align*}
 \sum_{Y\in \Theta} \bigg( \sum_{\substack{ X\in \Theta, \\ \dim(X)=k}} \bN_{X,Y} \bigg) \kappa(W,Y,m)
 =
 \sum_{Y\in \Theta} \binom{n-\dim Y}{n-k} \kappa(W,Y,m)
 =
 \sum_{j=0}^k
 \binom{n-j}{n-k}
 h_j.
\end{align*}
Indeed, the sum between parentheses gives the binomial coefficient (this is clear from the combinatorial definition), then we use equation~\eqref{eq:nar_krew} and the previous proposition. The relation we obtain is the first in~\eqref{eq:fh-transform} (the second one is obtained similarly).

Let us also briefly show how the bijection from Section~\ref{sec:bij} gives the formulas above for $h_k$ and~$h^+_k$.

\begin{proof}[Proof of Proposition~\ref{prop:hvectors}]
Via the bijection from Section~\ref{sec:bij}, faces of $\Upsilon(W,m)$ of dimension $k-1$ are in bijection with chains $w_0 \sqsubset w_1 \leq \dots \leq w_m$ in $\NC(W,c)$ such that $\ell(w_0)=n-k$. So, we get
\[
 \sum_{k=0}^n f_{k-1} z^{n-k}
 =
 \sum_{ w_0 \sqsubset w_1 \leq \dots \leq w_m }
 z^{\ell(w_0)}.
\]
Since the order ideal containing elements below $w_1$ for $\sqsubset$ is a boolean lattice, the binomial theorem gives:
\[
 \sum_{k=0}^n f_{k-1} z^{n-k}
 =
 \sum_{ w_1 \leq \dots \leq w_m }
 (1+z)^{\ell(w_1)}.
\]
By comparing with~\eqref{eq:fh-transform}, we immediately get the combinatorial formula for $h_k$. The one for $h^+_k$ is obtained similarly.
\end{proof}

It is well known that a shelling of a simplicial complex can be used to find its $h$-vector. We refer to~\cite{wachs}. Explicitly, suppose we have a shelling of $\Upsilon(W,m)$, i.e., an indexing $F_1,F_2,\dots$ of its facets such that for each $j$ the (geometric) intersection
\begin{align} \label{def:shelling}
 \Bigg( \bigcup_{i=1}^{j-1} F_i \Bigg) \cap F_j
\end{align}
is purely $1$-codimensional in $F_j$. Then $h_k$ is the number of indices $j$ such that the intersection in~\eqref{def:shelling} is the union of $k$ facets of $F_j$. It would be very interesting to show that our equation~\eqref{EQ:gammakappa} can also be obtained via such a shelling of $\Upsilon(W,m)$, and similarly equation~\eqref{EQ:gammakappa2} from a shelling of $\Upsilon^+(W,m)$. A natural candidate for this is the shelling given by Stump, Thomas and Williams~\cite{stumpthomaswilliams}.

From a different perspective, Reiner, Shepler, and Sommers~\cite{reinersheplersommers} have constructed invariant-theoretic $q$-analogs of the Kirkman and Narayana numbers for the coincidental types. They have proven a $q$-version of the transformation between $f$- and $h$-vectors for these $q$-Kirkman and $q$-Narayana numbers. However, that works mostly in the level of formula manipulations. It is an interesting open problem to extend this, perhaps by finding a more conceptual construction, to all real reflection groups.

\section{A recursion for the refined enumeration of faces}
\label{sec:recfaces}

We give a recursion satisfied by the numbers $\gamma(W , X , m)$, in the spirit of Fomin and Reading's recursion for the $f$-vector of $\Upsilon(W,m)$. Our main tool is also the rotation $\mathcal{R}_m$, and the method relies on the following key result.

\begin{Proposition} \label{prop:invariance}
 For each face $f \in \Upsilon(W,m)$, $\underline{f}$ and $\underline{\mathcal{R}_m(f)}$ are conjugate in $W$. In other words, $\Upsilon(W,X,m)$ is stable under the action of $\mathcal{R}_m$.
\end{Proposition}

\begin{proof}
 Let $f = \{ \rho_1 \succ \dots \succ \rho_k \}$, so that $\prod f = t_{\rho_1} \cdots t_{\rho_k} \in\NC(W,c)$ and $\underline{f} = c_+ (\prod f) c_-$. There is a factorization $\prod f = w_+ w_1 \cdots w_m w_-$ obtained as follows: $w_+$ (respectively, $w_i$, $w_-$) is obtained from $t_{\rho_1} \cdots t_{\rho_k}$ by keeping the factors $t_{\rho_j}$ such that $\rho_j\in -\Delta_+$ (respectively, such that $\rho_j$ is a positive root with color $i$, such that $\rho_j\in -\Delta_-$). This is possible because the definition of these factors agrees with the total order $\prec$ on $\Phi^{(m)}_{\geq-1}$: the elements in $-\Delta_+$ come first in $\rho_1,\rho_2,\dots$, etc. Using similar notations, the noncrossing partition associated to the face $f' = \mathcal{R}_m(f)$ is denoted $ \prod f' = w'_+ w'_1 \cdots w'_m w'_- $.

 To find what are the factors of $\prod f'$, we need to refine the factorization of $\prod f$. Note that from Proposition~\ref{prop:reforder}, we get
\[
 \{ \alpha_{nh/2-n+s+1} , \dots , \alpha_{nh/2} \} = c^{-1}(-\Delta_+).
\]
We write
\[
 w_m = w_{m,3} w_{m,2} w_{m,1},
\]
where $w_{m,3}$ contains the factors $t_{\alpha^m}$ where $\alpha \in \Delta_-$, $w_{m,2}$ contains the factors $t_{\alpha^m}$ where $\alpha \in c^{-1}(-\Delta_+)$, and $w_{m,1}$ contains the other factors $t_{\alpha^m}$. As above, this is possible because this definition agrees with the ordering on $\Phi^{(m)}_{\geq-1}$.

By examining the action of $\mathcal{R}_m$, we can check
\begin{align*}
 &w'_+ = c w_{m,2} c^{-1}, \\
 &w'_1= \big(c w_{m,1} c^{-1} \big) \big(c w_- c^{-1}\big) w_+, \\
 &w'_i= w_{i-1} \text{ for } 2\leq i \leq m,\\
 &w'_-= w_{m,3}.
\end{align*}
For example, $w'_- = w_{m,3}$ comes from the case $\mathcal{R}_m(\alpha^m) = -\alpha$ if $\alpha \in \Delta_-$.
Gathering the factors gives
\begin{align*}
 \prod f' &= \big( c w_{m,2} c^{-1} \big) \big( c w_{m,1} c^{-1} \cdot c w_- c^{-1} \cdot w_+ \big) w_1 \cdots w_{m-1} ( w_{m,3} ) \\
 &= \big( c w_{m,2} w_{m,1} w_- c^{-1} \cdot w_+ ) w_1 \cdots w_{m-1} w_{m,3} \\
 &= \big( c w_{m,2} w_{m,1} w_- c^{-1} \big)
 \Big( \prod f \Big)
 (w_{m,2} w_{m,1} w_-)^{-1}.
\end{align*}
Using $c=c_+c_-$ and the definition of $\underline{f}$, it follows
\[
 \underline{f'}
 =
 c_- (w_{m,2} w_{m,1} w_-) c_-
 \underline{f}
 c_- (w_{m,2} w_{m,1} w_-)^{-1} c_-.
\]
So $\underline{f'}$ and $\underline{f}$ are conjugate.
\end{proof}

{\samepage
\begin{Proposition} 
 If $W = W_1 \times W_2$ and $X = X_1 \times X_2 $, we have
\begin{align} \label{eq:gammareducible}
 \gamma(W, X, m )
 =
 \gamma(W_1, X_1, m ) \cdot \gamma(W_2, X_2, m ).
\end{align}
\end{Proposition}

\begin{proof}
This follows from $\Upsilon(W, X, m)$ being the join of $\Upsilon(W_1, X_1, m)$ and $\Upsilon(W_2, X_2, m)$.
\end{proof}}

To state the next proposition, we extend the definition of $\gamma(W, X, m)$ in the same way as we did for $\kappa$ and $\kappa^+$ in Section~\ref{sec:reciprocities}. We use parabolic conjugacy classes rather than flats, and if $\mathcal{Y} = \biguplus_{i=1}^j \mathcal{X}_i$ is a disjoint union of parabolic conjugacy classes, we write
\begin{align} \label{eq:gamma_uplus}
 \gamma(W, \mathcal{Y}, m)
 =
 \sum_{i=1}^j \gamma(W, \mathcal{X}_i, m).
\end{align}

\begin{Proposition} 
Assume that $W$ is irreducible and let $h$ be its Coxeter number. We have
\begin{equation}\label{eqrecfaces}
 \gamma(W , \mathcal{X}, m) = \frac{mh+2}{2k} \sum_{s\in S} \gamma( W_{(s)} , \mathcal{X}\cap W_{(s)} , m ).
\end{equation}
\end{Proposition}

\begin{proof}
Following the idea used in~\cite{fominreading} to count facets of $\Upsilon(W,m)$, we consider {\it pointed faces} and let
\[
 \Upsilon^\bullet(W, X, m)
 :=
 \big\{
 (f,\rho) \colon f\in \Upsilon(W, X, m),\;
 \rho \in f
 \big\}.
\]
The result comes from double counting. First, we clearly have
\[
 \# \Upsilon^\bullet(W, \mathcal{X}, m)
 =
 k \cdot \# \Upsilon(W, \mathcal{X}, m)
 =
 k\cdot \gamma(W, \mathcal{X}, m),
\]
as each face $f\in \Upsilon(W, \mathcal{X}, m)$ contains $k$ vertices.

Second, we consider the action of $\mathcal{R}_m$ on $\Upsilon^\bullet(W, \mathcal{X}, m)$ by $\mathcal{R}_m\big( (f,\rho) \big) = \big(\mathcal{R}_m(f),\mathcal{R}_m(\rho)\big)$. Let $\omega \subset \Phi_{\geq -1}^{(m)}$ be an orbit for the action of $\mathcal{R}_m$. We will show that
\begin{align} \label{equation_to_prove}
 \#\big\{ (f,\rho) \in \Upsilon^\bullet(W, \mathcal{X}, m) \colon \rho \in \omega \big\}
 =
 \frac{mh+2}{2}
 \sum_{ \alpha \in (-\Delta) \cap \omega }\gamma( W_{(t_\alpha)} , \mathcal{X}\cap W_{(t_\alpha)}, m ).
\end{align}
By summing over the orbits $\omega$, it will follow that $\# \Upsilon^\bullet(W, \mathcal{X}, m)$ is the right hand side of~\eqref{eqrecfaces} multiplied by $k$ (since the union of the sets $(-\Delta) \cap \omega$ is $-\Delta$). So it remains only to prove~\eqref{equation_to_prove}.

First observe that the cardinality in the left-hand side of~\eqref{equation_to_prove} actually doesn't depend on~$\rho$. Indeed, the rotation $\mathcal{R}_m$ immediately gives a bijection between the set associated to $\rho$ and that associated to $\mathcal{R}_m(\rho)$.

Consider the first case of Proposition~\ref{prop:mrotationorbits}. So, $\#\omega =\frac{mh+2}{2}$ and $\omega$ contains one element of~$-\Delta$, that we denote by $-\delta_i$. By the previous observation, we get
\begin{align*}
 \#\big\{ (f,\rho) \in \Upsilon^\bullet(W, \mathcal{X}, m) \colon \rho \in \omega \big\}
 =
 \frac{mh+2}{2}
 \#\big\{ f \in \Upsilon(W, \mathcal{X}, m) \colon {-}\delta_i \in f \big\}.
\end{align*}
The cardinality in the right-hand side is also $\gamma\big(W_{(s_i)}, \mathcal{X} \cap W_{(s_i)} ,m \big)$. We thus get the term indexed by $s_i$ in~\eqref{eqrecfaces} (up to the factor $k$).

Now, consider the second case of Proposition~\ref{prop:mrotationorbits}. So, $\#\omega =mh+2$ and $\omega$ contains two elements of $-\Delta$, that we denote by $-\delta_i$ and $-\delta_j$. In this case, we get
\begin{align*}
 \#\big\{ (f,\rho) \in \Upsilon^\bullet(W, \mathcal{X}, m) \colon \rho \in \omega \big\}
 =
 (mh+2) \cdot
 \#\big\{ f \in \Upsilon(W, \mathcal{X}, m) \colon {-}\delta_i \in f \big\}.
\end{align*}
Here, we get the two terms indexed by $s_i$ and $s_j$ in~\eqref{eqrecfaces} (up to the factor $k$). Indeed these two terms are equal (since $s_i$ and $s_j$ are conjugate), and combine to give one of the two terms multiplied by~$mh+2$.
\end{proof}

To finish this section, note that equations~\eqref{eq:gammareducible},~\eqref{eq:gamma_uplus}, and~\eqref{eqrecfaces} give an inductive procedure to compute $\gamma(W, X, m)$ as follows:
\begin{itemize}\itemsep=0pt
 \item if $\mathcal{X}$ is the conjugacy class of the Coxeter element, we have $\gamma(W, \mathcal{X},m) = 1$ (initial case),
 \item if $W$ is reducible, use~\eqref{eq:gammareducible},
 \item if $W$ is irreducible, use~\eqref{eqrecfaces}, then use~\eqref{eq:gamma_uplus} on each summand (in particular $\gamma\big(W_{(s)} , \mathcal{X}\cap W_{(s)}, m \big) = 0$ if $\mathcal{X}\cap W_{(s)}=\varnothing$),
 \item repeat the previous two steps, until each term can be treated by the initial case.
\end{itemize}

Also, it is interesting to make explicit what is the intersection $\mathcal{X}\cap W_{(s)}$ in the various cases of the finite type classification. The common situation is that $\mathcal{X}\cap W_{(s)}$ is itself a single conjugacy class of $W_{(s)}$. This comes from the fact that there is often a unique conjugacy class $\mathcal{X}$ of $W_{(s)}$ of a given Coxeter type (see tables in~\cite[Appendix~C]{orliksolomon}), as all elements in $\mathcal{X} \cap W_{(s)}$ have the same Coxeter type. The exceptions are given explicitly as follows.

First, let $W$ be of type $D_n$, where $n>4$ is odd. Assume that $\mathcal{X}$ contains elements of type~$A_{n-2}$, and $W_{(s)}$ is the standard parabolic subgroup of type $D_{n-1}$. There is a unique order~2 automorphism of the diagram of type $D_{n-1}$ (unless $n=5$, we leave details to the reader), and it defines an outer automorphism of $W_{(s)}$. Then $\mathcal{X}\cap W_{(s)}$ is the union of two conjugacy classes, which are image of each other under this outer automorphism.

In the case of $E_8$, there is a unique conjugacy class of type $A_5$, but the intersection with the parabolic subgroup of type $E_7$ is the union of two conjugacy classes. These are given the five-element subsets of the Dynkin diagram of $E_7$ as follows:
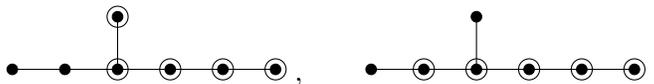
\[
 \begin{tikzpicture}[scale=0.7]
 \draw[fill=black] (0,0) circle (0.1cm);
 \draw[fill=black] (1,0) circle (0.1cm);
 \draw[fill=black] (2,0) circle (0.1cm);
 \draw[fill=black] (3,0) circle (0.1cm);
 \draw[fill=black] (4,0) circle (0.1cm);
 \draw[fill=black] (5,0) circle (0.1cm);
 \draw[fill=black] (2,1) circle (0.1cm);
 \draw (0,0) -- (1,0);
 \draw (1,0) -- (2,0);
 \draw (2,0) -- (3,0);
 \draw (3,0) -- (4,0);
 \draw (4,0) -- (5,0);
 \draw (2,0) -- (2,1);
 \draw (2,1) circle (0.2cm);
 \draw (2,0) circle (0.2cm);
 \draw (5,0) circle (0.2cm);
 \draw (4,0) circle (0.2cm);
 \draw (3,0) circle (0.2cm);
 \end{tikzpicture} \; ,\qquad
 \begin{tikzpicture}[scale=0.7]
 \draw[fill=black] (0,0) circle (0.1cm);
 \draw[fill=black] (1,0) circle (0.1cm);
 \draw[fill=black] (2,0) circle (0.1cm);
 \draw[fill=black] (3,0) circle (0.1cm);
 \draw[fill=black] (4,0) circle (0.1cm);
 \draw[fill=black] (5,0) circle (0.1cm);
 \draw[fill=black] (2,1) circle (0.1cm);
 \draw (0,0) -- (1,0);
 \draw (1,0) -- (2,0);
 \draw (2,0) -- (3,0);
 \draw (3,0) -- (4,0);
 \draw (4,0) -- (5,0);
 \draw (2,0) -- (2,1);
 \draw (1,0) circle (0.2cm);
 \draw (2,0) circle (0.2cm);
 \draw (5,0) circle (0.2cm);
 \draw (4,0) circle (0.2cm);
 \draw (3,0) circle (0.2cm);
 \end{tikzpicture} \; .
\]
The same phenomenon appears when $\mathcal{X}$ has type $A_1 \times A_3$, and then the two conjugacy classes in $E_7$ are given by
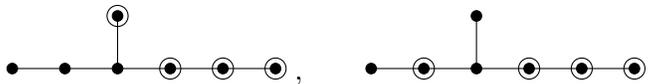
\[
 \begin{tikzpicture}[scale=0.7]
 \draw[fill=black] (0,0) circle (0.1cm);
 \draw[fill=black] (1,0) circle (0.1cm);
 \draw[fill=black] (2,0) circle (0.1cm);
 \draw[fill=black] (3,0) circle (0.1cm);
 \draw[fill=black] (4,0) circle (0.1cm);
 \draw[fill=black] (5,0) circle (0.1cm);
 \draw[fill=black] (2,1) circle (0.1cm);
 \draw (0,0) -- (1,0);
 \draw (1,0) -- (2,0);
 \draw (2,0) -- (3,0);
 \draw (3,0) -- (4,0);
 \draw (4,0) -- (5,0);
 \draw (2,0) -- (2,1);
 \draw (2,1) circle (0.2cm);
 \draw (5,0) circle (0.2cm);
 \draw (4,0) circle (0.2cm);
 \draw (3,0) circle (0.2cm);
 \end{tikzpicture} \; ,\qquad
 \begin{tikzpicture}[scale=0.7]
 \draw[fill=black] (0,0) circle (0.1cm);
 \draw[fill=black] (1,0) circle (0.1cm);
 \draw[fill=black] (2,0) circle (0.1cm);
 \draw[fill=black] (3,0) circle (0.1cm);
 \draw[fill=black] (4,0) circle (0.1cm);
 \draw[fill=black] (5,0) circle (0.1cm);
 \draw[fill=black] (2,1) circle (0.1cm);
 \draw (0,0) -- (1,0);
 \draw (1,0) -- (2,0);
 \draw (2,0) -- (3,0);
 \draw (3,0) -- (4,0);
 \draw (4,0) -- (5,0);
 \draw (2,0) -- (2,1);
 \draw (1,0) circle (0.2cm);
 \draw (5,0) circle (0.2cm);
 \draw (4,0) circle (0.2cm);
 \draw (3,0) circle (0.2cm);
 \end{tikzpicture} \; .
\]
Finally, this also occurs when $\mathcal{X}$ has type $A_1^3$, then the two conjugacy classes in $E_7$ are given by
\[
 \begin{tikzpicture}[scale=0.7]
 \draw[fill=black] (0,0) circle (0.1cm);
 \draw[fill=black] (1,0) circle (0.1cm);
 \draw[fill=black] (2,0) circle (0.1cm);
 \draw[fill=black] (3,0) circle (0.1cm);
 \draw[fill=black] (4,0) circle (0.1cm);
 \draw[fill=black] (5,0) circle (0.1cm);
 \draw[fill=black] (2,1) circle (0.1cm);
 \draw (0,0) -- (1,0);
 \draw (1,0) -- (2,0);
 \draw (2,0) -- (3,0);
 \draw (3,0) -- (4,0);
 \draw (4,0) -- (5,0);
 \draw (2,0) -- (2,1);
 \draw (2,1) circle (0.2cm);
 \draw (5,0) circle (0.2cm);
 \draw (3,0) circle (0.2cm);
 \end{tikzpicture} \; ,\qquad
 \begin{tikzpicture}[scale=0.7]
 \draw[fill=black] (0,0) circle (0.1cm);
 \draw[fill=black] (1,0) circle (0.1cm);
 \draw[fill=black] (2,0) circle (0.1cm);
 \draw[fill=black] (3,0) circle (0.1cm);
 \draw[fill=black] (4,0) circle (0.1cm);
 \draw[fill=black] (5,0) circle (0.1cm);
 \draw[fill=black] (2,1) circle (0.1cm);
 \draw (0,0) -- (1,0);
 \draw (1,0) -- (2,0);
 \draw (2,0) -- (3,0);
 \draw (3,0) -- (4,0);
 \draw (4,0) -- (5,0);
 \draw (2,0) -- (2,1);
 \draw (1,0) circle (0.2cm);
 \draw (5,0) circle (0.2cm);
 \draw (3,0) circle (0.2cm);
 \end{tikzpicture} \; .
\]

\section{The recursion for counting factorizations}
\label{sec:factor}

We consider a class of minimal factorizations of the Coxeter element and its refined enumeration via a $q$-statistic defined in terms of the orders $\sqsubset$ and $\ll$. This kind of refined enumeration was first considered in~\cite{josuatverges}, and reinterpreted in~\cite{bianejosuat}. In these references, we deal with reflection factorizations of the Coxeter element.

Here we consider more general factorizations where the first factor is in a given parabolic conjugacy class $\mathcal{X}$. They were considered by the first author in~\cite{douvropoulos}. As for the refined enumeration, the obtained polynomial $\phi(W,\mathcal{X},q)$ is shown to satisfy a recursion which is equivalent to the one obtained in Section~\ref{sec:recfaces} for $\gamma(W, X, m)$. Consequently, we get a formula for $\phi(W,\mathcal{X},q)$ in terms of the Orlik--Solomon exponents from the formula for $\gamma$ in Corollary~\ref{formulas_gammagammaplus}.

\begin{Definition}
Let $\mathcal{X}$ be a parabolic conjugacy class, and $k$ such that the elements of $\mathcal{X}$ have reflection length $n-k$. We define $\mathfrak{F}(W,\mathcal{X},c)$ as the set of length-additive factorizations $c=wt_1\cdots t_k$ where $w\in\mathcal{X}$ and $t_1,\dots,t_k \in T$. We define a statistic $\nil$ on this set by
\[
 \nil( w t_1 \cdots t_k )
 :=
 \# \big\{ i \in \{1,\dots,k\} \colon wt_1\cdots t_{i-1} \ll wt_1\cdots t_i \big\},
\]
and the associated generating function:
\[
 \phi(W,\mathcal{X},q)
 :=
 \sum_{ w t_1 \cdots t_k \in \mathfrak{F}(W , \mathcal{X} , c)}
 q^{\nil(wt_1\cdots t_k)}.
\]
\end{Definition}

Let us comment on this definition. First, note that the factorizations above are {\it length-additive}, i.e., $\ell(c) = \ell(w) + \ell(t_1) +\dots + \ell(t_k)$. Also, it is not {\it a priori} obvious from the definition that the generating function doesn't depend on the Coxeter element. This property will be apparent from the recursion given below (in particular it is part of the induction hypothesis). We anticipate this, omitting dependence on $c$ in the notation.

Since the factorizations are minimal, the two elements $wt_1\cdots t_{i-1}$ and $wt_1\cdots t_i$ in the definition of the $\nil$ statistic form a cover relation in $\NC(W,c)$. As these two elements differ by multiplying with a reflection, they are comparable in the {\it Bruhat order}, denoted $\leq_B$. We have
\begin{align*}
& wt_1\cdots t_{i-1} \ll wt_1\cdots t_i
 \quad\Longleftrightarrow\quad
 wt_1\cdots t_{i-1} \geq_B wt_1\cdots t_i, \\
& wt_1\cdots t_{i-1} \sqsubset wt_1\cdots t_i
\quad\Longleftrightarrow\quad
 wt_1\cdots t_{i-1} \leq_B wt_1\cdots t_i.
\end{align*}
Though it is interesting to make the connection with the order $\ll$ used through this work, here it will be adequate to think in terms of the Bruhat order and write
\[
 \nil( w t_1 \cdots t_k )
 =
 \# \big\{ i \in \{1,\dots,k\} \colon wt_1\cdots t_{i-1} \geq_B wt_1\cdots t_i \big\}.
\]

Let's recall some facts about {\it inversions}, which are closely connected to the Bruhat order. Recall from Section~\ref{sec:defclus} that $\Phi_+$ (respectively, $\Phi_-$) is the set of positive roots (respectively, negative roots). We call
\begin{itemize}\itemsep=0pt
 \item $\alpha\in\Phi_+$ an {\it inversion} of $w\in W$ if $w(\alpha)\in\Phi_-$,
 \item $t\in T$ a {\it right inversion} of $w\in W$ if $wt$ has smaller Coxeter length than $w$.
\end{itemize}
These two notions are related: $\alpha\in \Phi_+$ is an inversion of $w$ if and only if $t_\alpha$ is a right inversion of $w$.

\begin{Lemma} \label{inv_rem1}
There is a bijection between right inversions of $c$ and $S$. It sends $t$ to the unique $s\in S$ such that $ct \in W_{(s)}$.
\end{Lemma}

\begin{proof}
A reduced expression of $c$ is $c=s_1\cdots s_n$. By a well-known fact, it follows that $c$ has $n$ right inversions $t_1,\dots t_n$ explicitly given by
\[
 c t_i = (s_1 \cdots s_{i-1}) \cdot (s_{i+1} \cdots s_n).
\]
The result follows straightforwardly.
\end{proof}

\begin{Lemma} \label{lemm:orbits_roots}
 Assume that $W$ is irreducible, and let $h$ be its Coxeter number. Each orbit for the action of $c$ on $\Phi$ has cardinality $h$ and contains exactly one inversion of $c$.
\end{Lemma}

\begin{proof}
The fact that the orbits have cardinality $h$ follows from Steinberg's indexing of the root system (that we mentioned in Section~\ref{sec:defclus}, in relation with the reflection ordering). Though this concerns the bipartite Coxeter element, other standard Coxeter elements are conjugate to it and the result follows.

Next, we show that each orbit $\omega \subset \Phi$ contains at least one inversion of $c$. Since $c$ has $n$ inversions by the previous lemma, it will follow that each of the $n$ orbits contains exactly one inversion by the pigeonhole principle. So, let $\omega \subset \Phi$ be an orbit. Since $h$ is the order of $c$, we have
\[
 0 = c^h - I = (c-I)\Bigg(\sum_{i=0}^{h-1} c^i\Bigg).
\]
As $1$ is not an eigenvalue of $c$ (see~\cite[Section~3.19]{humphreys} for details), $c-I$ is invertible. Therefore,
\[
 \sum_{i=0}^{h-1} c^i = 0.
\]
By evaluating the linear operators on both sides of this equation at some $\alpha\in\omega$, we get ${\sum_{\alpha\in\omega}\alpha = 0}$. So $\omega$ contains both positive roots and negative roots. We deduce that $\omega$ contains at least a~positive root whose image by $c$ is negative, i.e., an inversion of $c$.
\end{proof}

The next proposition is similar to Proposition~\ref{prop:mrotationorbits} (and Steinberg's result described afterwards). But here, the statement holds for any standard Coxeter element and not just the bipartite Coxeter element.

\begin{Proposition}\label{propcconj}
Assume that $W$ is irreducible, and let $h$ be its Coxeter number. Then each orbit $\omega \subset T$ for the action of $c$ by conjugation satisfies
\begin{itemize}\itemsep=0pt
 \item either $\#\omega = h/2$ and it contains one right inversion of $c$,
 \item or $\#\omega = h$ and it contains two right inversions of $c$.
\end{itemize}
\end{Proposition}

\begin{proof}
The action of $c$ by conjugation on the set $T$ of reflections is related to the action of $c$ on roots by
\begin{equation*}
 t_{ c(\alpha) } = c \cdot t_\alpha \cdot c^{-1}.
\end{equation*}
So, this essentially follows from Lemma~\ref{lemm:orbits_roots}. Let $\omega \subset T$ be an orbit. Note that $R^{-1}(\omega) \subset \Phi$ is stable under the action of $c$ and under the map $-I$.

First suppose that $R^{-1}(\omega)$ is a single orbit under the action of $c$. Since the map $R$ is 2 to 1, $\omega$ has cardinality $h/2$. By Lemma~\ref{lemm:orbits_roots}, $R^{-1}(\omega)$ contains $1$ inversion of $c$, and it follows that $\omega$ contains one right inversion of $c$.

In the second case, suppose that $R^{-1}(\omega)$ is the union of two orbits. Since the map $R$ is 2 to~1, $\omega$ has cardinality $h$. By Lemma~\ref{lemm:orbits_roots}, $R^{-1}(\omega)$ contains $2$ inversion of $c$, and it follows that~$\omega$ contains two right inversion of $c$.
\end{proof}

To state the next result, we first need to slightly extend the definition of $\phi$. This is completely similar to what we did for $\gamma$ and $\kappa$. So, if $\mathcal{Y} = \biguplus_{i=1}^j \mathcal{X}_i$ is a disjoint union of parabolic conjugacy classes, we define
\begin{equation} \label{union_sum}
 \phi( W, \mathcal{Y} , q ) := \sum_{i=1}^j \phi( W, \mathcal{X}_i , q ).
\end{equation}
As before, this situation occurs when we consider the intersection of a parabolic conjugacy class with a parabolic subgroup.

\begin{Lemma} \label{lemm:induc}
 Let $t \in T$. There is an integer $i$ such that $c^i t c^{-i}$ is a right inversion of $c$, and this corresponds to an element $s\in S$ via the bijection in Lemma~$\ref{inv_rem1}$. With these notations, we have
 \begin{align} \label{lemma_sumt_k}
 \sum_{ \substack{ wt_1 \cdots t_k \in \mathfrak{F}(W,\mathcal{X},c), \\[1mm] t_k=t } }
 \!\! q^{\nil(wt_1 \cdots t_k )}
 =
 \begin{cases}
 \phi\big( W_{(s)}, \mathcal{X} \cap W_{(s)} , q \big) &\text{if $t$ is a right inversion of $c$}, \\
 q \phi\big( W_{(s)}, \mathcal{X} \cap W_{(s)} , q \big) &\text{otherwise}.
 \end{cases}
 \end{align}
\end{Lemma}

\begin{proof}
Suppose first that $t$ is a right inversion of $c$, so that $ct \in W_{(s)}$. It follows that $wt_1\cdots t_{k-1} \leq_B w t_1 \cdots t_k$ and this pair of elements doesn't contribute to the $\nil$ statistic. Since $ct$ is $s_1\cdots s_n$ with $s$ omitted (see the proof of Lemma~\ref{inv_rem1}), it is a standard Coxeter element of $W_{(s)}$. To each factorization $ c = w t_1 \cdots t_{k}$ as in the left-hand side of~\eqref{lemma_sumt_k}, we associate $ct = w t_1 \cdots t_{k-1}$ which is again a minimal factorization. We thus immediately obtain the generating function $\phi( W_{(s)}, \mathcal{X} \cap W_{(s)} , q )$.

Consider the second case. We now have $wt_1\cdots t_{k-1} \geq_B w t_1 \cdots t_k$, so that this pair of elements contribute by 1 to the $\nil$ statistic, giving the factor $q$. As in the previous case, the sum can be viewed as a sum over factorizations of $ct$. We thus get $q \phi( W_{\Fix(ct)}, \mathcal{X} \cap W_{\Fix(ct)} , q )$, but we have%
\[
 \phi\big( W_{\Fix(ct)}, \mathcal{X} \cap W_{\Fix(ct)} , q \big)
 =
 \phi\big( W_{(s)}, \mathcal{X} \cap W_{(s)} , q \big)
\]
since $W_{\Fix(ct)}$ is conjugate to $W_{(s)}$ in $W$.
\end{proof}

\begin{Proposition} 
Assume that $W$ is irreducible, and let $h$ be its Coxeter number. We have
\begin{equation} \label{eqrecfactors}
 \phi(W, \mathcal{X}, q) = \frac{2+q(h-2)}{2} \sum_{s\in S} \phi\big( W_{(s)} , \mathcal{X}\cap W_{(s)}, q \big).
\end{equation}
\end{Proposition}

\begin{proof}
We partition the set $\mathfrak{F}(W,\mathcal{X},c)$ into subsets $\mathfrak{F}_\omega$, where $\omega$ are the orbits for the action of $c$ on $T$,
\[
 \mathfrak{F}_\omega
 :=
 \{ wt_1\cdots t_k \in \mathfrak{F}(W,c,\mathcal{X}) \colon t_k \in \omega \},
\]
so that $\phi(W, \mathcal{X}, q)$ is obtained by summing the generating functions of the sets $\mathfrak{F}_\omega$.

We first consider the first case of Proposition~\ref{propcconj}, so that $\#\omega = \frac{h}{2}$ and $\omega$ contains a right inversion $t \in T$ of $c$. Let $s\in S$ such that $ct \in W_{(s)}$. Using Lemma~\ref{lemm:induc}, we get
\begin{align} \label{eq:omega_sum1}
 \sum_{ w t_1 \cdots t_k \in \mathfrak{F}_\omega } q^{\nil(wt_1 \cdots t_k)}
 =
 \big(1+q\big(\tfrac{h}{2}-1\big)\big) \phi\big( W_{(s)}, \mathcal{X} \cap W_{(s)} , q \big).
\end{align}

Now, we consider the other case and assume that $\#\omega = h$. By Proposition~\ref{propcconj}, this orbit contains two right inversions of $c$, that we denote $t,t' \in T$. Also, let $s,s'\in S$ such that $ct \in W_{(s)}$ and $ct' \in W_{(s')}$. Using Lemma~\ref{lemm:induc}, we get
\begin{align*}
 \sum_{ w t_1 \cdots t_k \in \mathfrak{F}_\omega } q^{\nil(wt_1 \cdots t_k)}
 =
 (2+q(h-2)) \phi\big( W_{(s)} , \mathcal{X}\cap W_{(s)}, q \big),
\end{align*}
which can also be written
\begin{align} \label{eq:omega_sum2}
 \sum_{ w t_1 \cdots t_k \in \omega }\!\! q^{\nil(wt_1 \cdots t_k)}
 \!=\!
 \big(1+q\big(\tfrac{h}{2}-1\big)\big)
 \big( \phi\big(W_{(s)} , \mathcal{X}\cap W_{(s)}, q \big)
 + \phi\big( W_{(s')} , \mathcal{X}\cap W_{(s')}, q \big)
 \big).
\end{align}
By summing~\eqref{eq:omega_sum1} and~\eqref{eq:omega_sum2}, we get the result. The fact that each $s\in S$ appears exactly once in the sum follows from Lemma~\ref{inv_rem1}.
\end{proof}

As for the reducible case, we have the following statement.

\begin{Proposition}
Suppose $W=W_1\times W_2$. For $i\in\{1,2\}$, let $\mathcal{X}_i$ be a parabolic conjugacy class in $W_i$. Let $k_i$ such that the elements of $\mathcal{X}_i$ have reflection length $n_i-k_i$ where $n_i$ is the rank of~$W_i$. Then
\begin{equation} \label{prop_redcase}
 \phi(W_1\times W_2, \mathcal{X}_1\times\mathcal{X}_2, q ) = \binom{k_1+k_2}{k_1}
 \phi(W_1,\mathcal{X}_1,q)
 \phi(W_2,\mathcal{X}_2,q).
\end{equation}
\end{Proposition}

\begin{proof}
The Coxeter element is denoted $c=(c_1,c_2)$. For $i \in \{1,2\}$, the projection $W_1\times W_2 \to W_i$ on each factor gives a map $\mathfrak{F}(W_1\times W_2, (c_1,c_2), \mathcal{X}_1 \times \mathcal{X}_2 ) \to \mathfrak{F}(W_i, c_i, \mathcal{X}_i )$ (ignoring the factors~$t_j$ that are not reflections $W_i$, that give the unit of $W_i$). Together, these two maps give one map
\[
 \mathfrak{F}(W_1\times W_2, \mathcal{X}_1 \times \mathcal{X}_2, (c_1,c_2) )
 \to
 \mathfrak{F}(W_1, \mathcal{X}_1, c_1 ) \times \mathfrak{F}(W_2, \mathcal{X}_2, c_2 ).
\]
It is easily seen that this map is $\binom{k_1+k_2}{k_1}$-to-$1$. Indeed, we recover the initial factorization $w t_1 \cdots t_{k_1+k_2}$ if we know which $t_i$ correspond to a reflection of $W_1$ (or $W_2$), and this gives a~choice of $k_1$ indices among $k_1+k_2$.

It is also easily seen that the $q$-statistic is preserved via the above map, as the Bruhat order on $W_1\times W_2$ identifies with the product of the two Bruhat orders.
\end{proof}

Gathering the previous propositions, we have an inductive way to compute $\phi(W, \mathcal{X}, q)$:
\begin{itemize}\itemsep=0pt
 \item if $\mathcal{X}$ is the conjugacy class of the Coxeter element, we have $\phi(W, \mathcal{X}, q) = 1$ (initial case),
 \item if $W$ is reducible, use~\eqref{prop_redcase},
 \item if $W$ is irreducible, use~\eqref{eqrecfactors}, then use~\eqref{union_sum} on each term,
 \item repeat the previous two steps until each term can be treated via the initial case.
\end{itemize}

\begin{Theorem}
Let $\mathcal{X}$ be a parabolic conjugacy class, and $k$ such that the elements of $\mathcal{X}$ have reflection length $n-k$. We have
 \begin{align} \label{eq:phigamma}
 \phi(W,\mathcal{X},q) = k! (1-q)^k \gamma\big( W , \mathcal{X} , \tfrac{q}{1-q} \big).
 \end{align}
\end{Theorem}

\begin{proof}
Up to rescaling of the polynomials, the induction to compute $\phi$ is equivalent to the one that we show $\gamma$ satisfies (see Section~\ref{sec:recfaces}). This shows the equality by induction.
\end{proof}

From the formula for $\gamma$ in terms of a characteristic polynomial, we immediately obtain the following formula for $\phi$. It is conveniently stated in terms of the Orlik--Solomon exponents.

\begin{Corollary}
Let $\mathcal{X}$ be a parabolic conjugacy class, let $X = \Fix(w)$ for some $w\in\mathcal{X}$, $k=\dim(X)$, and recall that $b_1^X,\dots, b_k^X$ are the Orlik--Solomon exponents of $X$. We have
\begin{align*}
 \phi(W,\mathcal{X},q)
 =
 \frac{k!}{[ N(W_X) : W_X ]} \prod_{i=1}^k \big( b_i^{X} + 1 + q\big( h - b_i^{X} - 1\big) \big).
\end{align*}
\end{Corollary}

\begin{proof}
This follows from~\eqref{eq:phigamma}, together with the formula for $\gamma$ in~\eqref{eq:formulagamma}.
\end{proof}

In particular, the case $q=1$ gives
\begin{align*}
 \phi(W,\mathcal{X},1)
 =
 \frac{k! h^k}{[ N(W_X) : W_X ]}.
\end{align*}
This was obtained by the first author~\cite[Theorem~99]{douvropoulos} by geometric methods (using the Lyashko--Looijenga morphism). On the other side, the case where $X$ is the maximal element of $L(W)$ corresponds to the $q$-enumeration of minimal reflection factorization of~$c$.  The second author in~\cite{josuatverges} (see also~\cite[Section~4.6]{bianejosuat}) showed that
\begin{align*}
 \phi(W,\{0\},1)
 =
 \frac{n!}{|W|} \prod_{i=1}^n (d_i + q(h-d_i))
\end{align*}
(where we use the {\it degree} $d_i=e_i+1$ rather the exponent $e_i$). This is a one-parameter refinement of the Deligne--Looijenga number
\[
 \frac{n! h^n}{|W|}.
\]

\section{Further comments and questions}

Via Proposition~\ref{prop:psik_expansion}, the formulas we obtained in Corollary~\ref{formulas_gammagammaplus} imply
\begin{align*} 
 &\sum_{X\in \Theta} (-1)^{\dim(X)} \gamma(W,X,m) \bPhi_X
=
 \bPsi_{-mh-1} = (-1)^n \sign \otimes \bPsi_{mh+1}, \\
 &\sum_{X\in \Theta} (-1)^{\dim(X)} \gamma^+(W,X,m) \bPhi_X
=
 \bPsi_{-mh+1} = (-1)^n \sign \otimes \bPsi_{mh-1}.
\end{align*}
This suggests a representation theoretic interpretation, which is the goal of our next paper~\cite{douvropoulosjosuatverges2}. There, we introduce {\it cluster parking functions} and their positive analog. They form a simplicial complex endowed with an action of $W$. The left-hand sides of the previous two equations correspond to their (equivariant) Euler characteristic. The main goal of the paper is to show that these complexes have the homotopy of (pure) wedge of spheres, so that the unique nonzero homology group have the right-hand sides above as characters (up to the sign $(-1)^n$).

In Corollary~\ref{formulas_gammagammaplus}, we give a refinement of the $f$-vector of the generalized cluster complex $\Upsilon(W,m)$, that looks \emph{almost} like a flag $f$-vector for it (as in \cite[Chapter~III, Section~4]{stanley}; almost because it is indexed by \emph{classes} of subsets of $[n]$). This is surprising because $\Upsilon(W,m)$ is not a~balanced complex. On the other side, the Coxeter complex of $W$ is balanced. It is natural to ask: is there a connection between these objects that can explain the appearance of this almost flag $f$-vector? Is this part of a more general theory for non-balanced complexes? We thank Vic Reiner for these questions.

In Proposition~\ref{prop:invariance}, we showed that $\Upsilon(W,X,m)$ is stable under the action of $\mathcal{R}_m$. It would be interesting to find a cyclic sieving phenomenon for this action (see~\cite{reinersommers} for similar and possibly related results).

\subsection*{Acknowledgements}

This project started in Paris when both authors first moved to IRIF and discovered they share a love for Coxeter--Catalan combinatorics. We thank Fr\'ed\'eric Chapoton for suggesting us to investigate the generating function $\phi$ (defined in Section~\ref{sec:factor}), which was a motivation for the whole project. We also thank Philippe Biane for our fruitful discussion throughout. Eventually, we thank the reviewers for their numerous suggestions that helped improving this article.

\pdfbookmark[1]{References}{ref}
\LastPageEnding

\end{document}